\documentclass[10pt]{amsart}
\usepackage{latexsym, amscd, amsfonts, eucal, mathrsfs, amsmath, amssymb, amsthm, tikz-cd, xr, stmaryrd, amsxtra, MnSymbol,bbm}
\usepackage[alphabetic]{amsrefs}
\usepackage{enumitem}
\usepackage[hidelinks]{hyperref}
\usepackage{chngcntr}

\counterwithin{equation}{subsection}
\newtheorem{theorem}[equation]{Theorem}
\newtheorem{lemma}[equation]{Lemma}
\newtheorem{lemma/construction}[equation]{Construction/Lemma}

\newtheorem{proposition}[equation]{Proposition}
\newtheorem{corollary}[equation]{Corollary}

\theoremstyle{definition}

\newtheorem{prop-con}[equation]{Proposition--Construction}
\newtheorem{notation}[equation]{Notation}

\newtheorem{construction}[equation]{Construction}

\newtheorem{definition}[equation]{Definition}
\newtheorem{definition/lemma}[equation]{Definition--Lemma}

\theoremstyle{remark}
\newtheorem{remark}[equation]{Remark}

\title{Potential diagonalisability of pseudo-Barsotti--Tate representations}

\author{Robin Bartlett}
\date{\today}
\begin{document}
	
	\maketitle
	\begin{abstract}
		Previous work of Kisin and Gee proves potential diagonalisability of two dimensional Barsotti--Tate representations of the Galois group of a finite extension $K/\mathbb{Q}_p$. In this paper we build upon their work by relaxing the Barsotti--Tate condition to one we call pseudo-Barsotti--Tate (which means that for certain embeddings $\kappa:K \rightarrow \overline{\mathbb{Q}}_p$ we allow the $\kappa$-Hodge--Tate weights to be contained in $[0,p]$ rather than $[0,1]$).
	\end{abstract}
	\tableofcontents

\section{Introduction}

\subsection{Overview} Following \cite[\S 1.4]{BLGGT}, a potentially crystalline representation of $G_K$ is potentially diagonalisable if, after restricting to $G_{K'}$ for some finite $K'/K$, it is contained in the same irreducible component of a crystalline deformation ring as a direct sum of characters.  In \cite{BLGGT} automorphy lifting theorems are proved for global representations which are potentially diagonalisable at places above $p$.

Unfortunately, potential diagonalisability has been established in only a small number of cases. If $K/\mathbb{Q}_p$ is unramified then crystalline representations with Hodge type in the Fontaine--Laffaille range are known to be potentially diagonalisable, cf. \cite{BLGGT} and \cite{GL14} for the extended Fontaine--Laffaille range. It is also known for ramified $K$ and Barsotti--Tate Hodge types (i.e. those concentrated in degrees $[0,1]$) by results in \cite{KisFF} and \cite{Gee06}. 

In \cite{B19} the author extended these results when $K/\mathbb{Q}_p$ is unramified to Hodge types concentrated in degrees $[0,p]$ (also a mild cyclotomic-freeness assumption is required). The aim of this paper is to show how the methods of \emph{loc. cit.} can also be applied when $K$ ramifies. The following are the precise assumptions we require:

\begin{definition}\label{introdef}
	Let $k$ denote the residue field of $K$ and choose an indexing $\kappa_{ij}$ for the embeddings $K \hookrightarrow \overline{\mathbb{Q}}_p$ so that $\kappa_{ij}|_k = \kappa_{i'j'}|_k$ if and only if $i=i'$. Also let $\mathbb{F}$ be a finite extension of $\mathbb{F}_p$.
	\begin{enumerate}
		\item A Hodge type $\mu = (\mu_{\kappa})_{\kappa:K\rightarrow \overline{\mathbb{Q}}_p}$ is pseudo-Barsotti--Tate if there exists such an indexing $\kappa_{ij}$ so that  $\mu_{\kappa_{ij}} \subset [0,h_j]$ with $h_1=p$ and $h_2=\ldots=h_e=1$.
		\item A continuous representation $V_{\mathbb{F}}$ of $G_K$ on an $\mathbb{F}$-vector space is cyclotomic-free if there exists an unramified extension $K'/K$ such that every Jordan--Holder factor $V$ of $V_{\mathbb{F}}|_{G_{K'}}$ is one-dimensional, and if $V$ is unramified then $V \otimes \mathbb{F}(-1)$ is not a Jordan--Holder factor  of $V_{\mathbb{F}}|_{G_{K'}}$.
	\end{enumerate}
\end{definition}

Thus, being pseudo--Barsotti--Tate is somewhere between being Barsotti--Tate and being concentrated in degrees $[0,p]$. The cyclotomic-freeness condition avoids possible extensions of the inverse of the cyclotomic character by the trivial representations. For example, the only non-cyclotomic-free two dimensional representations are of the form $\psi \otimes \left( \begin{smallmatrix}
	1 & * \\ &  \chi_{\operatorname{cyc}}^{-1} 
\end{smallmatrix}\right)$ for some unramified character $\psi$. Note also that cyclotomic-freeness depends only on the representations semi-simplifaction.

\begin{theorem}
	Every crystalline representation $V$ of $G_K$ with pseudo-Barsotti--Tate Hodge type and cyclotomic-free residual representation is potentially diagonalisable.
\end{theorem}

\subsection{Method}
The typical method for establishing potential diagonalisability is to replace $V$ by $V|_{G_K'}$ for $K'/K$ a sufficiently large unramified extension so that the residual representation becomes a successive extension of one-dimensional representations (such a $K'$ always exists after possibly extending the coefficient field). While $V$ may not itself be ordinary (that is, have every Jordan--Holder factor one-dimensional) one aims to produce an ordinary $V'$ lying on the same irreducible component in the crystalline deformation ring.

For $K/\mathbb{Q}_p$ unramified and Hodge types in the Fontaine--Laffaille range the key input which enables this approach is the observation that the deformation rings in question are formally smooth over $\mathbb{Z}_p$. Unfortunately, this is not the case once one leaves the Fontaine--Laffaille range. In \cite{B19} we addressed this by considering instead Kisin's ``resolution'' by moduli of Breuil--Kisin modules:
$$
\mathcal{L}^\mu_{R_{V_{\mathbb{F}}}} \rightarrow \operatorname{Spec}R^\mu_{V_{\mathbb{F}}}
$$ 
The key calculation was that for $\mu$ concentrated in degrees $[0,p]$, $\mathcal{L}^\mu$ is formally smooth over $\mathbb{Z}_p$. As this morphism becomes an isomorphism after inverting $p$, potentialy diagonalisability can then be established by arguing as in the previous paragraph, but with $V$ and $V'$ replaced with points in $\mathcal{L}^\mu_{R_{V_{\mathbb{F}}}}$, i.e. after replacing $V$ and $V'$ by their corresponding Breuil--Kisin modules.

When $K/\mathbb{Q}_p$ ramifies the situation is worse still. Even in the case of Barsotti--Tate $\mu$ considered in \cite{KisFF} the $\mathcal{L}^\mu_{R_{V_{\mathbb{F}}}} $ have normal special fibres but need not be smooth. This normality is sufficient to establish potential diagonalisability in some cases via an explicit construction of paths between points in these spaces. However, this involves some laborious computations. The key idea in this paper is recover smoothness by replacing $\mathcal{L}^\mu_{R_{V_{\mathbb{F}}}} $ by a futher ``Demazure'' type resolution 
$$
\mathcal{L}^{\mu,\operatorname{conv}}_{R_{V_{\mathbb{F}}}}  \rightarrow \mathcal{L}^\mu_{R_{V_{\mathbb{F}}}} 
$$
classifying Breuil--Kisin modules together with a specific filtration $\mathcal{F}^\bullet$ on the image of its Frobenius. The key technical result is then:

\begin{theorem}\label{smooth}
	Assume that $\mu$ is pseudo--Barsotti--Tate and $V_{\mathbb{F}}$ is cyclotomic-free (in fact a weaker condition suffices here). Then 
	$$
	\mathcal{L}^{\mu,\operatorname{conv}}_{R_{V_{\mathbb{F}}}} \rightarrow \operatorname{Spec}R^\mu_{V_{\mathbb{F}}}
	$$
	becomes an isomorphism after inverting $p$ and the local rings of $\mathcal{L}^{\mu,\operatorname{conv}}_{R_{V_{\mathbb{F}}}}$ are formally smooth over $\mathbb{Z}_p$ at closed points.
\end{theorem}

Once we have this theorem potential diagonalisability follows by an essentially identical argument to that employed in \cite{B19}.

The proof of Theorem~\ref{smooth} is based on a tangent space calculation; we show that at any closed point the tangent space of the special fibre is $\leq$ the dimension of generic fibre. Since the generic fibre identifies with the generic fibre of $\operatorname{Spec}R^\mu_{V_{\mathbb{F}}}$ the latter value is well-known. To bound the mod~$p$ tangent space we observe that since $\mathcal{L}^{\mu,\operatorname{conv}}_{R_{V_{\mathbb{F}}}}$ is $\mathbb{Z}_p$-flat by definition any mod~$p$ tangent vector is induced from an $A$-valued point for $A$ some finite flat $\mathbb{Z}_p$-algebra. Such an $A$-valued point corresponds to a filtered Breuil--Kisin module attached to a crystalline representation on a finite $A$-module. Forgetting the $A$-action we use a generalisation to dimensions $>2$ of a result from \cite{GLS15} to ensure that the reduction modulo $p$ of this filtered Breuil--Kisin module is of a specific form (this is where the restriction to pseudo-Barsotti--Tate Hodge types is crucial). Computing the possible extensions of filtered Breuil--Kisin modules of this specific form produces the desired bound.

\subsection*{Acknowledgements}

I would like to thank the Max Planck Institute for Mathematics for its support during the writing of this paper.

\section{Notation}\label{notation}

\subsection{General conventions}
Throughout we let $K$ denote a finite extension of $\mathbb{Q}_p$ with residue field a degree $f$ extension $k$ of $\mathbb{F}_p$. Let $e$ denote the ramification degree of $K$ over $\mathbb{Q}_p$ and fix a uniformiser $\pi \in K$. Let $G_K$ denote the absolute Galois group of $K$. We write $E(u) \in W(k)[u]$ for the minimal polynomial of $\pi$ over $W(k)$. This is a degree $e$ polynomial with $E(u) \equiv u^e$ modulo $\pi$. We also fix a compatible system $\pi^{1/p^\infty}$ of $p$-th power roots of $\pi$ inside a completed algebraic closure $C$ of $K$. We set $K_\infty = K(\pi^{1/p^\infty})$. When $p=2$ we additionally require that $\pi$ be as in the following lemma (when $p>2$ this condition is automatic).

\begin{lemma}\label{1.1.1}
	If $p=2$ then there exists a uniformiser $\pi \in K$ so that $K_\infty \cap K(\mu_{p^\infty}) = K$; here $\mu_{p^\infty}$ denotes the group of $p$-th power roots of unity in $C$.
\end{lemma} 
\begin{proof}
	See \cite[2.1]{Wang17}.
\end{proof}
\subsection{Coefficients}\label{coeff}
We also fix a finite extension $E$ of $\mathbb{Q}_p$ with ring of integers $\mathcal{O}$ and residue field $\mathbb{F}$. These play the role of coefficient rings. We assume that $E$ contains a Galois closure of $K$ so that there are $ef$ distinct embeddings $K\hookrightarrow E$. It will be convenient to choose an indexing $\kappa_{ij}$ of these embeddings by $1\leq i\leq f$ and $1 \leq j \leq e$ as in Definition~\ref{introdef} so that 
$$
\kappa_{ij}|_k = \kappa_{i'j'}|_k \Leftrightarrow i=i'
$$
There is an isomorphism $K \otimes_{\mathbb{Q}_p} E \cong \prod_{ij} E$ given by $a\otimes b \mapsto (\kappa_{ij}(a)b)_{ij}$. This allows us to decompose any $K\otimes_{\mathbb{Q}_p} E$-module $M$ as $\prod_{ij} M_{ij}$ where $M_{ij}$ is the submodule of $M$ on which $K$ acts via $\kappa_{ij}$.

We emphasise that the identification $K\otimes_{\mathbb{Q}_p} E \cong \prod_{ij} E$ does not descend to $\mathcal{O}_K \otimes_{\mathbb{Z}_p} \mathcal{O}$ because the idempotents in $K\otimes_{\mathbb{Q}_p} E$ involve non-integral terms (the only exception being when $K=K_0$). However, we do have a similar decomposition $W(k)\otimes_{\mathbb{Z}_p}\mathcal{O} \cong \prod_i \mathcal{O}$ given by $a\otimes b \mapsto (\kappa_i(a)b)_i$ where $\kappa_i = \kappa_{ij}|_{W(k)}$ (which by construction is independent of $j$). Thus, every $W(k)\otimes_{\mathbb{Z}_p} \mathcal{O}$-module $M$ can be decomposed as $M = \prod_i M_i$ where $M_i$ denotes the submodule of $M$ on which $W(k)$ acts through $\kappa_i$. 

In particular this allows us to refine the construction of $E(u) \in W(k)[u]$ as follows. Define $E_{ij}(u) \in W(k)[u]\otimes_{\mathbb{Z}_p} \mathcal{O}$ as the element corresponding to
$$
(\ldots0,\underbrace{u - \kappa_{ij}(\pi)}_{i-th\text{ position}},0\ldots) \in \prod_i \mathcal{O}[u]
$$
under the identification $W(k)[u] \otimes_{\mathbb{Z}_p} \mathcal{O} \cong \prod_i \mathcal{O}[u]$. We also set $E_j(u) = \prod_{i=1}^f E_{ij}(u)$. Note that $\prod_{ij} E_{ij}(u) = \prod_j E_j(u) = E(u)$.

\subsection{Filtered modules}\label{filt}
A filtered module $M$ over a ring $A$ is a finite $A$-module equipped with a filtration 
$$
\ldots \subset \operatorname{Fil}^{i+1}(M) \subset \operatorname{Fil}^i(M) \subset \ldots  
$$
by $A$-submodules of $M$ with $A$-projective graded pieces and with $\operatorname{Fil}^n(M) = 0$ for $n>>0$ and $\operatorname{Fil}^n(M) = M$ for $n<<0$. If $\lambda$ is a multiset of integers then we say $M$ has type $\lambda$ is $\operatorname{gr}^n(M)$ has constant rank equal to the multiplicity of $n$ in $\lambda$. If $M'$ is another filtered $A$-module write $\operatorname{Hom}_{\operatorname{Fil}}(M,M')$ for the module of $A$-linear homomorphisms $M\rightarrow M'$ equipped with the filtration 
$$
\operatorname{Fil}^n(\operatorname{Hom}_{\operatorname{Fil}}(M,M')) = \lbrace x:M\rightarrow M'\mid x(\operatorname{Fil}^n(M)) \subset \operatorname{Fil}^n(M')\rbrace
$$ 
If $M'$ has type $\lambda'$ then
$$
d(M,M') = \operatorname{rank}_A \frac{\operatorname{Hom}_{\operatorname{Fil}}(M,M')}{\operatorname{Fil}^0(\operatorname{Hom}_{\operatorname{Fil}}(M,M'))}
$$
is equal to the $\sum_{x \in \lambda} \operatorname{Card}(\lbrace x'\in \lambda' \mid x>x'\rbrace)$. In particular, it depends only on $\lambda$ and $\lambda'$ and we write $d(M,M') = d(\lambda,\lambda')$. 

\subsection{Hodge types}\label{Hodgetype}
A Hodge type $\mu$ is a tuple $(\mu_{ij})$ indexed by $1\leq i \leq f, 1\leq j\leq e$ of multisets of integers (all of the same cardinality). The decompositions from Section~\ref{coeff} allows us to produce, from either of the following two sets of data,
\begin{enumerate}
	\item A tuple $D_1,\ldots,D_e$ of filtered $k\otimes_{\mathbb{F}_p} \mathbb{F}$-modules.
	\item A filtered $K\otimes_{\mathbb{Q}_p} E$-module.
\end{enumerate}
a tuple of filtered vector spaces indexed by $1\leq i \leq f,1\leq j \leq e$. We say that objects as in either (1) or (2) have Hodge type $\mu$ if the $ij$-th filtered vector space has type $\mu_{ij}$. We also write
$$
d(\mu,\mu') = \sum_{ij} d(\mu_{ij},\mu'_{ij})
$$

\subsection{Period rings}
Let $\mathcal{O}_{C^\flat}$ denote the inverse limit of the system 
$$
\mathcal{O}_C/p \leftarrow \mathcal{O}_C/p \leftarrow \ldots
$$
whose transition maps are given by $x \mapsto x^p$. This is a domain in characteristic $p$ equipped with an action of $G_K$ induced by that on $\mathcal{O}_C/p$. Its field of fractions $C^\flat$ is algebraically closed (and identifies non-canonically with the completed algebraic closure of $k((u))$). Hence $A_{\operatorname{inf}} = W(\mathcal{O}_{C^\flat})$ and $W(C^\flat)$ admit $G_K$-actions as well as the Witt vector Frobenius. The compatible system $\pi^{1/p^\infty}$ gives rise to an element $\pi^\flat \in \mathcal{O}_{C^\flat}$. Via this choice we embed $\mathfrak{S} = W(k)[[u]] \rightarrow A_{\operatorname{inf}}$ by $u \mapsto [\pi^\flat]$. This embedding is $\varphi$-equivariant when $\mathfrak{S}$ is equipped with the Frobenius which on $W(k)$ is the Witt vector Frobenius and which sends $u \mapsto u^p$. It is also $G_{K_\infty}$-equivariant when $\mathfrak{S}$ is equipped with the trivial $G_{K_\infty}$-action.

\subsection{Crystalline representations}\label{crys}
A continuous representation of $G_K$ on a finite dimensional $E$-vector space $V$ is crystalline if $D_{\operatorname{crys}}(V) := (V \otimes_{\mathbb{Q}_p} B_{\operatorname{crys}})^{G_K}$ has $K_0$-dimension equal to the $\mathbb{Q}_p$-dimension of $V$. In this case $D_{\operatorname{crys}}(V)$ is a finite free $K_0 \otimes_{\mathbb{Q}_p} E$-module of rank equal to the $E$-dimension of $V$. We write $D_{\operatorname{crys},K}(V) = D_{\operatorname{crys}}(V) \otimes_{K_0} K$ which is equipped with the filtration given by
$$
\operatorname{Fil}^nD_{\operatorname{crys},K}(V) = (V \otimes_{\mathbb{Q}_p} t^iB_{\operatorname{dR}}^+)^{G_K}
$$
and we say $V$ has Hodge type $\mu$ is $D_{\operatorname{crys},K}(V)$ has Hodge type $\mu$. Note that our normalisations are such that the cyclotomic-character has Hodge type $-1$.
\section{Moduli of Breuil--Kisin modules}\label{moduli}

\subsection{Basic definitions}

For any $\mathcal{O}$-algebra $A$ set $\mathfrak{S}_A = \mathfrak{S} \otimes_{\mathbb{Z}_p} A$ and write $\varphi$ for the $A$-linear extension of $\varphi$ on $\mathfrak{S}$. Recall also the elements $E_{j}(u) = \prod_{i=1}^f E_{ij}(u)$ in $W(k)[u] \otimes_{\mathbb{Z}_p} \mathcal{O} \subset \mathfrak{S}_{\mathcal{O}}$ from Section~\ref{coeff}. In this paper only the case of $A$ finite over $\mathcal{O}$ will be relevant.

\begin{definition}
	Consider integers $h_j \geq 0$ for $j=1,\ldots,e$. A Breuil--Kisin module over $A$ of height $\leq h_j$ is a finite projective $\mathfrak{S}_A$-module $\mathfrak{M}$ equipped with an $\mathfrak{S}_A$-linear homomorphism
	$$
	\mathfrak{M} \otimes_{\mathfrak{S},\varphi} \mathfrak{S} \rightarrow \mathfrak{M}
	$$
	with cokernel killed by $\prod_{j=1}^e E_j(u)^{h_j}$. 
\end{definition}

For any such $\mathfrak{M}$ write $\mathfrak{M}^\varphi$ for the image of this homomorphism and $\varphi(\mathfrak{M})$ for the image of the composite
$\mathfrak{M} \rightarrow \mathfrak{M} \otimes_{\mathfrak{S},\varphi} \mathfrak{S} \rightarrow \mathfrak{M}$ with the first map given by $m \mapsto m \otimes 1$. Note that $\varphi(\mathfrak{M})$ is a $\varphi(\mathfrak{S}_A)$-module which generates $\mathfrak{M}^\varphi$ over $\mathfrak{S}_A$.

\begin{lemma}\label{projective}
 	If $A$ is $\mathcal{O}$-finite and $\mathfrak{M}$ is a Breuil--Kisin module over $A$ of height $\leq h_j$ then both $\mathfrak{M}/\mathfrak{M}^\varphi$ and $\mathfrak{M}^\varphi/\prod E_j(u)^{h_j}\mathfrak{M}$ are finite projective $A$-modules.
\end{lemma}
\begin{proof}
From the exact sequence
	$$
	0 \rightarrow \mathfrak{M}^\varphi /(\prod_{i=1}^e E_j(u)^{h_j})\mathfrak{M} \rightarrow \mathfrak{M}/(\prod_{i=1}^e E_j(u)^{h_j})\mathfrak{M} \rightarrow \mathfrak{M}/\mathfrak{M}^\varphi \rightarrow 0
	$$
	it is enough to consider $\mathfrak{M}/\mathfrak{M}^\varphi$. That this is finite projective over $A$ is proven in \cite[3.4.1]{CEGS19}.
\end{proof}
\subsection{Recalling a construction of Kisin}\label{recallkisin}
Now fix a continuous representation of $G_K$ on a finite dimension $\mathbb{F}$-vector space $V_{\mathbb{F}}$ together with a choice of $\mathbb{F}$-basis and let $R = R_{V_{\mathbb{F}}} = R_{V_{\mathbb{F}}}^\square \otimes_{W(k)} \mathcal{O}$ denote the corresponding $\mathcal{O}$-framed deformation ring. Write $V_R$ for the universal deformation and for any homomorphism $\alpha:R \rightarrow A$ write $V_A = V_{\alpha} =V_R \otimes_{\alpha, R} A$.

\begin{construction}
	For each $h_j\geq 0$ and each Artin $\mathcal{O}$-algebra with finite residue field set $\mathcal{L}^{\leq h_j}(A)$ equal to the set pairs $(\mathfrak{M},\alpha)$ where $\alpha: R \rightarrow A$ is a homomorphism and $\mathfrak{M}$ is a  $G_{K_\infty}$-stable $\mathfrak{S}_A$-submodule of $V_A \otimes_{\mathbb{Z}_p} W(C^\flat)$ with
	\begin{equation}\label{ident}
	\mathfrak{M} \otimes_{\mathfrak{S}}W(C^\flat) = V_A \otimes_{\mathbb{Z}_p} W(C^\flat)
\end{equation}
	and for which the semilinear extension of the trivial Frobenius on $V_A$ makes $\mathfrak{M}_A$ into a Breuil--Kisin module over $A$ of height $\leq h_j$.
	Base-change along $A$ make this into a functor on Artin $\mathcal{O}$-algebras.
\end{construction}

For any Hodge type $\mu$ let $R^{\mu}$ denote the unique $\mathcal{O}$-flat reduced quotient of $R$ with the property that a homomorphism $R \rightarrow B$ into a finite $E$-algebra factors through $R^\mu$ if and only if $V_B$ is crystalline of Hodge type $\mu$. The existence of such a quotient is the main result of \cite{Kis08}. If we assume $\mu$ is concentrated in degrees $[0,h_j]$ then have the following:

\begin{proposition}[Kisin]\label{kisin}
	The functor $A \mapsto \mathcal{L}^{\leq h_j}(A)$ is represented by a scheme $\mathcal{L}_R^{\leq h_j}$ and the morphism $\Theta:\mathcal{L}^{\leq h_j}_R \rightarrow \operatorname{Spec}R$ given by $(\mathfrak{M},\alpha) \mapsto \alpha$ is projective. Furthermore $\Theta[\frac{1}{p}]$ is a closed immersion and $\operatorname{Spec}R^\mu \rightarrow \operatorname{Spec}R$ factors through the scheme-theoretic image of $\Theta$.
\end{proposition}
\begin{proof}
	When each $h_j =h$ then this follows from \cite{Kis08} (in particular see 1.5.1 and 1.6.4 of \emph{loc. cit.}). The construction of $\mathcal{L}^{\leq h}_R$ also shows that for any $h_j \leq h$ both $A\mapsto (\prod_{i=1}^e E_j(u)^{h_j})\mathfrak{M}/E(u)^h\mathfrak{M}$ and $A\mapsto \mathfrak{M}/\mathfrak{M}^\varphi$ extend to shaves of $\mathcal{O}_{\mathcal{L}^{\leq h}}\otimes_{\mathbb{Z}_p} \mathfrak{S}$-modules which are coherent as $\mathcal{O}_{\mathcal{L}^{\leq h}_R}$-modules. Lemma~\ref{projective} shows they are also locally free. As a consequence, the locus of $\mathcal{L}^{\leq h}_R$ over which $(\prod_{i=1}^e E_j(u)^{h_j}) \mathfrak{M} \subset \mathfrak{M}^\varphi$ is closed. This is precisely $\mathcal{L}^{\leq h_j}_R$. It remains only to show that if $\mu$ is concentrated in degrees $[0,h_j]$ then it factors through $\mathcal{L}^{\leq h_j}_R$. This follows from part (4) of Lemma~\ref{filproperties} which is proven in Section~\ref{filonphi}.
\end{proof}

An easy limit argument shows that the description of the $A$-points of $\mathcal{L}_R^{\leq h_j}$ is valid whenever $A$ is a finite $\mathcal{O}$-algebra (i.e. not necessarily Artinian).

\subsection{A convolution variant} 
	We now produce a variant of $\mathcal{L}^{\leq h_j}_R$. For any Artin $R$-algebra $A$ set $\mathcal{L}^{\leq h_j,\operatorname{conv}}(A)$ equal to the set of triples $(\mathfrak{M},\alpha,\mathcal{F})$ for which $(\mathfrak{M},\alpha) \in \mathcal{L}^{\leq h_j}(A)$ and $\mathcal{F}$ is a sequence of $\mathfrak{S}_{A}$-submodules
	\begin{equation}\label{filtration}
	\left( \prod_{j=1}^e E_j(u)^{h_j} \right) \mathfrak{M} = \mathcal{F}^e \subset \ldots \subset \mathcal{F}^1 \subset \mathcal{F}^0 = \mathfrak{M}^\varphi
\end{equation}
	with $E_j(u)^{h_j}\mathcal{F}^j \subset \mathcal{F}^j \subset \mathcal{F}^{j-1}$ and $\mathcal{F}^{i-1}/\mathcal{F}^i$ finite projective over $A$ for each $j$.

\begin{proposition}\label{conv}
	The functor $A \mapsto \mathcal{L}^{\leq h_j,\operatorname{conv}}(A)$ is represented by a scheme $\mathcal{L}^{\leq h_j,\operatorname{conv}}_R$. Furthermore, the morphism $\mathcal{L}^{\leq h_j,\operatorname{conv}}_R \rightarrow \mathcal{L}^{\leq h_j}_R$ given by $(\mathfrak{M},\alpha,\mathcal{F}) \mapsto (\mathfrak{M},\alpha)$ is projective and, after inverting $p$, becomes a closed immersion through which the reduced subscheme of $\mathcal{L}^{\leq h_j}_{R}[\frac{1}{p}]$ factors.
\end{proposition}
\begin{proof}
	The representability of $\mathcal{L}^{\leq h_j,\operatorname{conv}}$ follows from the observation made in the proof of Proposition~\ref{kisin} that $A\mapsto \mathfrak{M}^\varphi/(\prod E_j(u)^{h_j})\mathfrak{M}$ extends to a coherent locally free sheaf on $\mathcal{L}^{\leq h_j}_R$. Indeed, this shows that $\mathcal{L}^{\leq h_j,\operatorname{conv}}$ can be constructed as a succession of  extensions of Grassmannians over $\mathcal{L}^{\leq h_j}_R$.
	
	To show that $\mathcal{L}^{\leq h_j,\operatorname{conv}}_R \rightarrow \mathcal{L}^{\leq h_j}_R$ becomes a closed immersion after inverting $p$ it suffices, since this morphism is proper, to show the induced map on $B$-valued points is injective for any finite $E$-algebra $B$. As explained in e.g. the first paragraph of the proof of \cite[1.6.4]{Kis08}, any $B$-valued point is induced from an $A$-valued point for some finite flat $\mathcal{O}$-algebra $A$. Thus, it suffices to show that for any such $A$ and any $A$-valued point $(\mathfrak{M},\alpha)$ of $\mathcal{L}^{\leq h_j}_R$ there exists at most one filtration $\mathcal{F}$ as in \eqref{filtration}. Since $\mathfrak{S}_E$ is a principal ideal domain and the $E_j(u)$ are pairwise coprime 
	$$
	\mathfrak{S}_E/(\prod E_j(u)^{h_j}) \cong \prod \mathfrak{S}_E /E_j(u)^{h_j}
	$$ 
	From this it is immediate that after inverting $p$ a unique such filtration $\mathcal{F}'$ exists. Then $\mathcal{F}^i \subset \mathcal{F}'^i$ and $\mathcal{F}^i[\frac{1}{p}] =\mathcal{F}'$. This implies $\mathcal{F}^i = \mathcal{F}'^i \cap \mathfrak{M}^\varphi$, since $\mathfrak{M}^\varphi/\mathcal{F}^i$ is $p$-torsionfree, and so $\mathcal{F}$ is unique as claimed. 
	
	Note that if $A$ is the ring of integers in a finite extension of $E$ then $\mathcal{F}'^i \cap \mathfrak{M}^\varphi$ defines a filtration as in \eqref{filtration} because the graded pieces, being $p$-torsionfree, are $A$-projective (this need not be the case for more general $A$). This shows that $\mathcal{L}^{\leq h_j,\operatorname{conv}}_R \rightarrow \mathcal{L}^{\leq h_j}_R$ induces a bijection on $A$-valued points for such $A$. It follows from e.g. \cite[7.2.5]{LLBBM20} that $\mathcal{L}^{\leq h_j,\operatorname{conv}}_R \rightarrow \mathcal{L}^{\leq h_j}_R$ induces an isomorphism on underlying reduced subschemes.
\end{proof}

\begin{remark}
	In \cite[2.2.16]{B19} it is explained how $\mathcal{L}^{\mu}_R \rightarrow \operatorname{Spec}R^\mu$ induces a bijection on $\mathcal{O}'$-valued points when $\mathcal{O}'$ is the ring of integers inside a finite extension $E'/E$. The last paragraph of the previous proof shows that the same is true with $\mathcal{L}^\mu_R$ replaced by $\mathcal{L}^{\mu,\operatorname{conv}}_R$. We emphasise that this is not true when $\mathcal{O}'$ is an arbitrary finite flat $\mathcal{O}$-algebra (otherwise $\mathcal{L}^{\mu,\operatorname{conv}}_R \rightarrow \operatorname{Spec}R^\mu$ would be an isomorphism).
\end{remark}
Since $\operatorname{Spec}R^{\mu}[\frac{1}{p}]$ is reduced Proposition~\ref{conv} implies that it can be viewed as a closed subscheme of $\mathcal{L}^{\leq h_j,\operatorname{conv}}_R[\frac{1}{p}]$. This allows us to make the following definition.
\begin{definition}
	For any Hodge type $\mu$ concentrated in degrees $[0,h_j]$ define $\mathcal{L}^{\mu,\operatorname{conv}}_R$ as the closure of $\operatorname{Spec}R^{\mu}[\frac{1}{p}]$ in $\mathcal{L}^{\leq h_j,\operatorname{conv}}_R$ 
\end{definition}

The main object of this paper is to describe the local geometry of $\mathcal{L}^{\mu,\operatorname{conv}}_R$ under the assumptions from Definition~\ref{introdef}.

\begin{theorem}[Main theorem]\label{main}
	Assume that 
	\begin{itemize}
		\item $\mu$ is pseudo-Barsotti--Tate, i.e. that $\mu$ is concentrated in degrees $[0,h_j]$ for $h_1=p$ and $h_2=\ldots=h_e =1$
		\item For any $G_{K_\infty}$-stable submodule $V \subset V_{\mathbb{F}}$ which is unramified there exists no $G_{K_\infty}$-equivariant quotient $V_{\mathbb{F}} \rightarrow W$ with $W \cong V \otimes \mathbb{F}(-1)$.
	\end{itemize}
	Then the local rings of $\mathcal{L}^{\mu,\operatorname{conv}}$ at closed points are formally smooth over $\mathbb{Z}_p$. 
\end{theorem}

In Lemma~\ref{cyclofree=>} we explain why condition (2) is satisfied whenever $V_{\mathbb{F}}$ is cyclotomic-free in the sense of Definition~\ref{introdef}.
\begin{proof}[Proof granting the results of Section~\ref{tgt}]
	Let $x \in \mathcal{L}^{\mu,\operatorname{conv}}_R$ any closed point. Enlarging $\mathbb{F}$ if necessary we can assume that $x$ is an $\mathbb{F}$-valued point. We show in Proposition~\ref{bound} below that the tangent space of $\mathcal{L}^{\mu,\operatorname{conv}}_R \otimes_{\mathcal{O}} \mathbb{F}$ at $x$ has dimension 
	$$
	\leq d^2 + d(\mu,\mu)
	$$
	(recall $d(\mu,\mu)$ is the value described in Section~\ref{hodgetypes}). On the other hand, in \cite[3.3.8]{Kis08} it is shown that $R^\mu[\frac{1}{p}]$ is equidimensional of dimension $d^2+d(\mu,\mu)$. The same is therefore true of $\mathcal{L}^{\mu,\operatorname{conv}}_R[\frac{1}{p}]$ and so, by flatness, also for $\mathcal{L}^{\mu,\operatorname{conv}}_R \otimes_{\mathcal{O}} \mathbb{F}$. This shows that the above inequality is an equality and that the local rings of $\mathcal{L}^{\mu,\operatorname{conv}}_R \otimes_{\mathcal{O}} \mathbb{F}$ at closed points are regular. Since the local rings of $\mathcal{L}^{\mu,\operatorname{conv}}_R$ are $\mathbb{Z}_p$-flat by definition it follows from \cite[07NQ]{stacks-project} that they are formally smooth over $\mathbb{Z}_p$.
\end{proof}
\begin{remark}
	One could also consider the closure of $\operatorname{Spec}R^{\mu}[\frac{1}{p}]$ in $\mathcal{L}^{\leq h_j}$. However the schemes obtained in this way typically fail to be regular. For example, it is shown in \cite{KisFF} that when $\mu$ is concentrated in degrees $[0,1]$ these spaces are smoothly equivalent to certain local models defined in \cite{PR09} which, while normal, are not necessarily smooth.
\end{remark}
\section{Strongly divisible extensions}

\subsection{Categories of mod~$p$ Breuil--Kisin modules}\label{categories}

Let $\operatorname{Mod}_{\mathcal{F}}$ denote the category whose objects are pairs $(\mathfrak{M},\mathcal{F})$ with $\mathfrak{M}$ a Breuil--Kisin module over $\mathbb{F}$ of any height $\geq 0$ and $\mathcal{F}$ a sequence of $\mathfrak{S}_{\mathbb{F}}$-submodules
$$
u^{e+p-1}\mathfrak{M} = \mathcal{F}^e  \subset \ldots \subset \mathcal{F}^1 \subset \mathcal{F}^0 = \mathfrak{M}^\varphi
$$
satisfying $u^p\mathcal{F}^0 \subset \mathcal{F}^1 \subset \mathcal{F}^0$ and $u\mathcal{F}^{i-1} \subset \mathcal{F}^i \subset \mathcal{F}^{i-1}$ for $i=2,\ldots,e$. Morphisms $\operatorname{Hom}_{\mathcal{F}}(\mathfrak{P},\mathfrak{M})$ in this category are $\varphi$-equivariant maps of $\mathfrak{S}_{\mathbb{F}}$-modules respecting the filtrations.

\begin{definition}\label{SD}
	Let $\operatorname{Mod}^{\operatorname{SD}}_{\mathcal{F}}$ denote the full subcategory of $\operatorname{Mod}_{\mathcal{F}}$ consisting of those $(\mathfrak{M},\mathcal{F})$ for which there exists an $\mathbb{F}_p[[u^p]]$-basis $(e_i)$ of $\varphi(\mathfrak{M})$ and integers $r_i\in [0,p]$ such that $\mathcal{F}^1$ is generated by $(u^{r_i}e_i)$.
\end{definition}

\begin{remark}
	When $e=1$ the category $\operatorname{Mod}^{\operatorname{SD}}_{\mathcal{F}}$ is precisely the category denoted in $\operatorname{Mod}^{\operatorname{SD}}_k$ and studied in \cite{B18}.
\end{remark}

The following is another interpretation of $\operatorname{Mod}^{\operatorname{SD}}_{\mathcal{F}}$. For $(\mathfrak{M},\mathcal{F}) \in \operatorname{Mod}_{\mathcal{F}}$ define 
\begin{itemize}
	\item $\operatorname{Fil}^i(\mathfrak{M}^\varphi) = \mathfrak{M}^\varphi \cap u^i \mathcal{F}^1$.
	\item $\operatorname{Fil}^i(\overline{\mathfrak{M}})$ equal to the image of $\operatorname{Fil}^i(\mathfrak{M}^\varphi)$ in $\overline{\mathfrak{M}} = \mathfrak{M}^\varphi /u\mathfrak{M}^\varphi$.
	\item $F^i(\mathfrak{M}) = \varphi(\mathfrak{M}) \cap u^i\mathcal{F}^1$
\end{itemize}

\begin{lemma}\label{equivalent}
	$(\mathfrak{M},\mathcal{F}) \in \operatorname{Mod}^{\operatorname{SD}}_{\mathcal{F}}$ if and only if image of the composite
	$$
	F^i(\mathfrak{M}) \rightarrow \mathfrak{M}^\varphi \rightarrow \overline{\mathfrak{M}}
	$$
	is $\operatorname{Fil}^n(\overline{\mathfrak{M}})$ for all $n$.
\end{lemma}
\begin{proof}
	If $(\mathfrak{M},\mathcal{F}) \in \operatorname{Mod}^{\operatorname{SD}}_{\mathcal{F}}$ then choose $(e_i)$ and $(r_i)$ as in Definition~\ref{SD}. We see that $\operatorname{Fil}^n(\mathfrak{M}^\varphi)$ is generated over $\mathbb{F}_p[[u]]$ by $u^{\operatorname{min}\lbrace n+r_i,0\rbrace}e_i$. Therefore $\operatorname{Fil}^n(\overline{\mathfrak{M}})$ is generated by the images of those $e_i$ with $n+r_i \leq 0$. This shows that the image of the composite in the lemma surjects onto $\operatorname{Fil}^n(\overline{\mathfrak{M}})$.
	
	For the converse, choose an $\mathbb{F}_p$-basis of $\overline{\mathfrak{M}}$ adapted to the filtration, i.e. choose a basis $(\overline{e}_i)$ and integers $(r_i)$ so that $\operatorname{Fil}^n(\overline{\mathfrak{M}})$ is generated by those $\overline{e}_i$ for which $n+r_i \leq 0$. By assumption we can find $e_i \in \varphi(\mathfrak{M}) \cap u^{r_i}\mathcal{F}^1$ whose image in $\overline{\mathfrak{M}}$ is $\overline{e}_i$. Clearly any such choice of $e_i$ produces an $\mathbb{F}_p[[u^p]]$-basis of $\varphi(\mathfrak{M})$. To see that $\mathcal{F}^1$ is generated by $u^{r_i}e_i$ we argue by (decreasing) induction on the smallest integer $n$ such that $u^nx \in \mathfrak{M}^\varphi$. Since $u^nx \in \operatorname{Fil}^n(\mathfrak{M}^\varphi)$ we can write
	$$
	u^nx \equiv\sum_{r_i+n\leq 0} \alpha_i e_i ~\operatorname{mod}~ u\mathfrak{M}^\varphi
	$$
	and so $x \equiv \sum_{r_i \leq -n} \alpha_i u^{-n}e_i~ \operatorname{mod}~ u^{1-n}\mathfrak{M}^\varphi$. Using the inductive hypothesis we deduce $x$ can be expressed as an $\mathbb{F}_p[[u]]$-linear combination of the $u^{r_i}e_i$ as claimed.
\end{proof}
\subsection{Properties of exact sequences in $\operatorname{Mod}^{\operatorname{SD}}_{\mathcal{F}}$}

We say that a sequence of morphisms
$$
0 \rightarrow (\mathfrak{M},\mathcal{E}) \rightarrow( \mathfrak{N},\mathcal{F}) \rightarrow (\mathfrak{P},\mathcal{G}) \rightarrow 0
$$
in $\operatorname{Mod}_{\mathcal{F}}$ is exact if the induced sequences $0 \rightarrow \mathcal{E}^i \rightarrow \mathcal{F}^i \rightarrow \mathcal{G}^i \rightarrow 0$ are exact for all $i$ when viewed as sequences of $\mathfrak{S}_{\mathbb{F}}$-modules.

\begin{proposition}\label{SDext}
	Suppose $(\mathfrak{N},\mathcal{F}) \in \operatorname{Mod}^{\operatorname{SD}}_{\mathcal{F}}$.
	\begin{enumerate}
		\item Then $(\mathfrak{M},\mathcal{E})$ and $(\mathfrak{P},\mathcal{G}) \in \operatorname{Mod}_{\mathcal{F}}^{\operatorname{SD}}$.
		\item The induced sequences $0 \rightarrow \operatorname{gr}^i(\overline{\mathfrak{M}}) \rightarrow \operatorname{gr}^i(\overline{\mathfrak{N}}) \rightarrow \operatorname{gr}^i(\overline{\mathfrak{P}}) \rightarrow 0$ are exact for each $i$.
		\item There exists an $\mathfrak{S}_{\mathbb{F}}$-linear splitting $s:\mathfrak{P}\rightarrow \mathfrak{N}$ such that $s(\mathcal{G}^1) \subset \mathcal{F}^1$ and $s(\varphi(\mathfrak{P})) \subset \varphi(\mathfrak{N})$.
	\end{enumerate}
\end{proposition}

Parts (1) and (2) of this proposition were proved in \cite{B18} in the case $e=1$. For the general case we observe that the condition in Definition~\ref{SD} only refers to the relative positions of $\mathcal{F}^1$ and $\varphi(\mathfrak{M})$; in particular it is a condition on the image of the Frobenius morphism rather than the morphism itself.

\begin{proof}
	To reduce to the case $e=1$ we produce a commutative diagram
	$$
	\begin{tikzcd}[column sep = tiny]
		0 \arrow[r] & u^{-p}\mathcal{E}^1 \otimes_{\varphi} \varphi(\mathfrak{S}_{\mathbb{F}}) \arrow[r] \arrow[d] &u^{-p}\mathcal{F}^1 \otimes_{\varphi} \varphi(\mathfrak{S}_{\mathbb{F}})\arrow[d] \arrow[r] & u^{-p}\mathcal{G}^1 \otimes_{\varphi} \varphi(\mathfrak{S}_{\mathbb{F}}) \arrow[d] \arrow[r] & 0 \\
		0 \arrow[r] &\varphi(\mathfrak{M}) \arrow[r] &\varphi(\mathfrak{N}) \arrow[r] & \varphi(\mathfrak{P})  \arrow[r] & 0
	\end{tikzcd}
	$$
	with the vertical arrows being isomorphisms of $\varphi(\mathfrak{S}_{\mathbb{F}})$-modules. As the vertical arrows go between projective $\varphi(\mathfrak{S}_{\mathbb{F}})$-modules of the same rank the outer arrows can be chosen arbitrarily and, after choosing an $\varphi(\mathfrak{S}_{\mathbb{F}})$-linear splitting of the top exact sequence, this determines the central vertical arrow. In the language of \cite{B18}, this makes $0 \rightarrow u^{-p}\mathcal{E}^1\rightarrow u^{-p}\mathcal{F}^1\rightarrow u^{-p}\mathcal{G}^i\rightarrow 0$ into an exact sequence in $\operatorname{Mod}^{\operatorname{BK}}_k$ with $u^{-p}\mathcal{F}^1$ strongly divisible (as in \cite[5.2.9]{B18}). Applying \cite[5.4.6]{B18} we deduce (1) and (2).
	
	Now we address (3). Note that (2) ensures we can choose a $k\otimes_{\mathbb{F}_p} \mathbb{F}$-linear splitting $\overline{s}$ of $\overline{\mathfrak{N}} \rightarrow \overline{\mathfrak{P}}$ mapping $\operatorname{Fil}^i(\overline{\mathfrak{P}})$ into $\operatorname{Fil}^i(\overline{\mathfrak{N}})$ by choosing successive splittings of the surjections $\operatorname{gr}^i(\overline{\mathfrak{N}}) \rightarrow \operatorname{gr}^i(\overline{\mathfrak{P}})$. From $\overline{s}$ we obtain a splitting
	$$
	s: \varphi(\mathfrak{P}) = \overline{\mathfrak{P}} \otimes_{k\otimes_{\mathbb{F}_p}\mathbb{F}} \varphi(\mathfrak{S}_{\mathbb{F}}) \rightarrow \overline{\mathfrak{N}}\otimes_{k\otimes_{\mathbb{F}_p}\mathbb{F}} \varphi(\mathfrak{S}_{\mathbb{F}}) = \varphi(\mathfrak{N})
	$$
	of $\varphi(\mathfrak{N})\rightarrow \varphi(\mathfrak{P})$ which we claim maps $\mathcal{G}^1$ into $\mathcal{F}^1$. For this we first show that
	$$
	s(F^n(\mathfrak{P})) \subset F^n(\mathfrak{N})
	$$ 
	for $n\leq p$ As $\overline{s}$ is compatible with the filtrations on $\overline{\mathfrak{P}}$ and $\overline{\mathfrak{N}}$ it follows that the image of $F^n(\mathfrak{P}) := \varphi(\mathfrak{P})\cap u^n\mathcal{G}^1$ under $s$ in $\overline{\mathfrak{N}}$ is contained in $\operatorname{Fil}^i(\overline{\mathfrak{N}})$. Lemma~\ref{equivalent} therefore implies that if $x \in F^n(\mathfrak{P})$ then $s( x) = x_1+u^px_2$ for $x_1 \in F^n(\mathfrak{N})$ and $x_2 \in \varphi(\mathfrak{N})$. If $n \leq p$ then $u^px_2 \in F^n(\mathfrak{N})$ and so $s(x) \in F^n(\mathfrak{N})$ also. This establishes the displayed inclusion above. To show $s(\mathcal{G}^1)\subset \mathcal{F}^1$ note that by (1) we have $\mathfrak{P} \in \operatorname{Mod}^{\operatorname{SD}}_{\mathcal{F}}$ and so there is a basis $(e_i)$ of $\varphi(\mathfrak{P})$ as in Definition~\ref{SD}. Since $e_i \in u^{-r_i}\mathcal{G}^1$ we have $s(e_i) \in u^{-r_i}\mathcal{F}^1$ and so
	$$
	s(u^{r_i}e_i) =u^{r_i}s(e_i) \in \mathcal{F}^1
	$$
	which finishes the proof.
\end{proof}

\subsection{Ext group via an explicit complex}

For $\mathfrak{M} = (\mathfrak{M},\mathcal{E})$ and $\mathfrak{P} = (\mathfrak{P},\mathcal{G})$ in $\operatorname{Mod}_{\mathcal{F}}$ consider the first Yoneda extension group $\operatorname{Ext}^1_{\mathcal{F}}(\mathfrak{P},\mathfrak{M})$ in $\operatorname{Mod}_{\mathcal{F}}$, i.e. the set of exact sequences 
$$
0 \rightarrow \mathfrak{M} \rightarrow( \mathfrak{N},\mathcal{F}) \rightarrow \mathfrak{P} \rightarrow 0,\qquad (\mathfrak{N},\mathcal{F}) \in \operatorname{Mod}_{\mathcal{F}}
$$
as considered in the previous section, modulo the equivalence relation identifying two sequences if and only if there exists a morphism $\alpha$ in $\operatorname{Mod}_{\mathcal{F}}$ making the following diagram commute.
$$
\begin{tikzcd}
	0 \ar[r] & \mathfrak{M} \ar[r] \ar[d] & (\mathfrak{N},\mathcal{F})\ar[r] \ar[d,"\alpha"] & \mathfrak{P}\ar[r] \ar[d] & 0 \\
	0 \ar[r] & \mathfrak{M}\ar[r]  & (\mathfrak{N}',\mathcal{F}')\ar[r]  & \mathfrak{P}\ar[r] & 0 
\end{tikzcd}
$$
We define $\operatorname{Ext}^1_{\operatorname{SD}}(\mathfrak{P},\mathfrak{M}) \subset \operatorname{Ext}^1_{\mathcal{F}}(\mathfrak{P},\mathfrak{M})$ as the subset consisting of those classes which can be represented by exact sequences as above with $(\mathfrak{N},\mathcal{F}) \in \operatorname{Mod}^{\operatorname{SD}}_{\mathcal{F}}$. Proposition~\ref{SDext} implies that $\operatorname{Ext}^1_{\operatorname{SD}}(\mathfrak{P},\mathfrak{M})$ is empty unless $\mathfrak{P},\mathfrak{M} \in \operatorname{Mod}^{\operatorname{SD}}_{\mathcal{F}}$.
\begin{notation}
In what follows, for any pair of $\mathfrak{S}_{\mathbb{F}}$-modules, we write $\operatorname{Hom}(-,-)$ for the set of $\mathfrak{S}_{\mathbb{F}}$-linear maps. In the special case of $\operatorname{Hom}(\mathfrak{P},\mathfrak{M})$ we equip this set with the Frobenius given by $\varphi(h) = \varphi_{\mathfrak{M}} \circ h \circ \varphi^{-1}_{\mathfrak{P}}$.
\end{notation}
	For $\mathfrak{M}$ and $\mathfrak{P}$ as above (we emphasis that for this definition we do not require $\mathfrak{M},\mathfrak{P} \in \operatorname{Mod}^{\operatorname{SD}}_{\mathcal{F}}$) consider the submodule 
	$$
	\mathcal{C}^1_{\operatorname{SD}} \subset  \prod_{1}^{e-1} \operatorname{Hom}(\mathfrak{P},\mathfrak{M})[\tfrac{1}{u}]
	$$
	consisting of those $(g_1,\ldots,g_{e-1})$ satisfying $g_i(\mathcal{G}^{i+1}) \subset \mathcal{E}^{i}$ and $g_i(u\mathcal{G}^{i}) \subset \mathcal{E}^{i+1}$ for each $i$.
	This fits into a two-term complex
	$$
	\begin{aligned}
	\mathcal{C}_{\operatorname{SD}}: F^0(\operatorname{Hom}(\mathfrak{P},\mathfrak{M})) \times \prod_{i=\operatorname{min}\lbrace e,2\rbrace}^{e-1} \operatorname{Hom}(\mathcal{G}^i,\mathcal{E}^i)  \xrightarrow{d_{\operatorname{SD}}} \mathcal{C}_{\operatorname{SD}}^1\\
	(h_1,\ldots,h_{e-1}) \mapsto (h_2-h_1,h_3-h_2,\ldots,\varphi^{-1}(h_1)-h_{e-1})
\end{aligned}
	$$
	where, as for objects in $\operatorname{Mod}_{\mathcal{F}}$, we write $F^i(\operatorname{Hom}(\mathfrak{P},\mathfrak{M}))$ for the set of $h \in \varphi(\operatorname{Hom}(\mathfrak{P},\mathfrak{M}))$ mapping $\mathcal{G}^1$ into $u^i\mathcal{E}^1$ 
	\begin{remark}
		When $e=1$ the above formula does not strictly make sense. Instead $d_{\operatorname{SD}}$ should be defined as $d_{\operatorname{SD}}(h_1) = \varphi^{-1}(h_1)-h_1$. In this case we recover the complex considered in \cite[4.1]{B18b}.
	\end{remark}. The following lemma motivates the construction of $\mathcal{C}_{\operatorname{SD}}$.

\begin{lemma}
	$H^0(\mathcal{C}_{\operatorname{SD}}) = \operatorname{Hom}_{\mathcal{F}}(\mathfrak{P},\mathfrak{M})$ and there is an injection 
	$$
	\operatorname{Ext}^1_{\operatorname{SD}}(\mathfrak{P},\mathfrak{M}) \rightarrow H^1(\mathcal{C}_{\operatorname{SD}})
	$$
	(in fact it is a bijection if $\mathfrak{P},\mathfrak{M} \in \operatorname{Mod}^{\operatorname{SD}}_{\mathcal{F}}$ we but only require injectivity for our applications).
\end{lemma}
\begin{proof}
	The first assertion is easy so we focus on the second. To construct the injection begin by considering an exact sequence $0 \rightarrow \mathfrak{M} \rightarrow( \mathfrak{N},\mathcal{F}) \rightarrow \mathfrak{P} \rightarrow 0$ in $\operatorname{Mod}^{\operatorname{SD}}_{\mathcal{F}}$. For each $i$ we can choose $\mathfrak{S}_{\mathbb{F}}$-linear splittings $s_i$ of $\mathcal{F}^i \rightarrow \mathcal{G}^i$ for $i=1,\ldots,e-1$. Proposition~\ref{SDext} allows us to assume that $s_1$ maps $\varphi(\mathfrak{P})$ into $\varphi(\mathfrak{N})$. As such $\varphi^{-1}(s_1)$ maps $\mathfrak{P}$ into $\mathfrak{M}$ and so also $\mathcal{G}^e$ into $\mathcal{F}^e$. It is then immediate that
	$$
	g = (s_2-s_1,s_3-s_2,\ldots,\varphi^{-1}(s_1)-s_{e-1})
	$$
	defines an element in $\mathcal{C}_{\operatorname{SD}}^1$. Suppose that $s'_i$ is another choice of splittings with corresponding $g' \in \mathcal{C}^1_{\operatorname{SD}}$. Then $s'_i - s_i \in \operatorname{Hom}(\mathcal{G}^i,\mathcal{E}^i)$ and $s_1 -s_1' \in F^0(\operatorname{Hom}(\mathfrak{P},\mathfrak{M})))$. This shows that $g-g'$ is contained in the image of $d_{\operatorname{SD}}$ so we obtain a well-defined element of $H^1(\mathcal{C}_{\operatorname{SD}})$.
	
	Now suppose $0 \rightarrow \mathfrak{M} \rightarrow( \mathfrak{N},\mathcal{F}) \rightarrow \mathfrak{P} \rightarrow 0$ and $0 \rightarrow \mathfrak{M} \rightarrow( \mathfrak{N}',\mathcal{F}') \rightarrow \mathfrak{P} \rightarrow 0$ are exact sequences mapping, by the construction from the previous subsection, into the same element $H^1(\mathcal{C}_{\operatorname{SD}})$. Then each exact sequence admits splittings $s_i$ and $s_i'$ so that
	$$
	(s_2-s_1,s_3-s_2,\ldots,\varphi^{-1}(s_1)-s_{e-1}) = (s_2'-s_1',s_3'-s_2',\ldots,\varphi^{-1}(s_1')-s_{e-1}')
	$$ 
	Equivalently 
	$$
	s_1-s_1' = s_2-s_2' = \ldots = s_{e-1}-s_{e-1}' = \varphi^{-1}(s_1)-\varphi^{-1}(s_1')
	$$
	For $n \in \mathfrak{N}$ write $\overline{n}$ for its image in $\mathfrak{P}$ and consider the map $\alpha:\mathfrak{N} \rightarrow \mathfrak{N}'$ given by
	$$
	n \mapsto n-s_1(\overline{n})+s_1'(\overline{n})
	$$
	Note this makes sense because $n-s_1(\overline{n}) \in \mathcal{E}^1$ and so can be viewed as an element of $\mathcal{F}'^1$. The fact that $s_1-s_1' = \varphi^{-1}(s_1)-\varphi^{-1}(s_1')$ shows this map is $\varphi$-equivariant. The fact that $s_1-s_1' = s_i-s_i'$ implies $\mathcal{F}^i$ is mapped into $\mathcal{F}'^i$. Therefore $\alpha$ is a morphism in $\operatorname{Mod}^{\operatorname{SD}}_{\mathcal{F}}$ which shows our two exact sequences define the same element in $\operatorname{Ext}^1_{\operatorname{SD}}(\mathfrak{P},\mathfrak{M})$. As a consequence the construction from the first paragraph produces an injection $\operatorname{Ext}^1_{\operatorname{SD}}(\mathfrak{P},\mathfrak{M})\rightarrow H^1(\mathcal{C}_{\operatorname{SD}})$ as desired.
\end{proof}

\subsection{Dimension calculations}
For $\mathfrak{M}$, $\mathfrak{P} \in \operatorname{Mod}_{\mathcal{F}}$ write 
$$
\operatorname{Hom}(\mathfrak{P},\mathfrak{M})_k := \varphi(\operatorname{Hom}(\mathfrak{P},\mathfrak{M}))/u^p
$$ 
and let $F^i(\operatorname{Hom}(\mathfrak{P},\mathfrak{M})_k)$ denote the image of $F^i(\operatorname{Hom}(\mathfrak{P},\mathfrak{M}))$ in $\operatorname{Hom}(\mathfrak{P},\mathfrak{M})_k$.

\begin{proposition}\label{extdim}
	The cohomology of $\mathcal{C}_{\operatorname{SD}}$ is $\mathbb{F}$-finite and if 
	$$
	\chi(\mathfrak{P},\mathfrak{M}) := \operatorname{dim}_{\mathbb{F}} H^1(\mathcal{C}_{\operatorname{SD}}) - \operatorname{dim}_{\mathbb{F}} H^0(\mathcal{C}_{\operatorname{SD}})
	$$
	then
	$$
	\chi(\mathfrak{P},\mathfrak{M}) = \operatorname{dim}_{\mathbb{F}}\frac{\operatorname{Hom}(\mathfrak{P},\mathfrak{M})_k}{F^0(\operatorname{Hom}(\mathfrak{P},\mathfrak{M})_k)} + \sum_{i=1}^{e-1} \operatorname{dim}_{\mathbb{F}} \operatorname{Hom}(\mathcal{G}^{i+1}/u\mathcal{G}^{i},\mathcal{E}^{i}/\mathcal{E}^{i+1})
	$$
	provided that $\operatorname{gr}(\operatorname{Hom}(\mathfrak{P},\mathfrak{M})_k) =0$ for $i< -p$.
\end{proposition}

We begin by proving the proposition under the following assumptions: (i) every $h \in \varphi(\operatorname{Hom}(\mathfrak{P},\mathfrak{M}))$ maps $\mathcal{G}^i$ into $\mathcal{E}^i$ for every $i$ and (ii) $\varphi(\operatorname{Hom}(\mathfrak{P},\mathfrak{M})) \subset u\operatorname{Hom}(\mathfrak{P},\mathfrak{M})$. 

\begin{lemma}\label{firstcase}
	Proposition~\ref{extdim} holds under assumptions (i) and (ii).
\end{lemma}
\begin{proof}
	Assumption (ii) implies that $u$-adically $\varphi^n(h) \rightarrow 0$ for every $h \in \operatorname{Hom}(\mathfrak{P},\mathfrak{M})$. From this we deduce $\varphi-1$ is an $\mathbb{F}$-linear automorphism of $\operatorname{Hom}(\mathfrak{P},\mathfrak{M})$. Injectivity is clear and, for surjectivity, if $h \in \operatorname{Hom}(\mathfrak{P},\mathfrak{M})$ then $\varphi-1$ sends $-\sum_{n \geq 0} \varphi^n(h)$ onto $h$. From injectivity of $\varphi-1$ we deduce that $H^0(\mathcal{C}_{\operatorname{SD}}) = 0$.
	
	Since $H^0(\mathcal{C}_{\operatorname{SD}}) = 0$ the map $d_{\operatorname{SD}}:\mathcal{C}_{\operatorname{SD}}^0 \rightarrow \mathcal{C}^1_{\operatorname{SD}}$ is injective. However, assumption (i) also allows us to view $\mathcal{C}^0_{\operatorname{SD}} \subset \mathcal{C}^1_{\operatorname{SD}}$ via obvious map $(h_i)\mapsto (h_i)$, Furthermore, the cokernel of this second inclusion naturally identifies with
	$$
	\mathcal{H}^1 := \prod_{i = 1}^{e-1} \operatorname{Hom}(\mathcal{G}^{i+1}/u\mathcal{G}_{i},\mathcal{E}^{i}/\mathcal{E}^{i+1})
	$$
	To relate the cokernel of $d_{\operatorname{SD}}$ with $\mathcal{H}^1$ we refine $\mathcal{C}^0_{\operatorname{SD}} \subset \mathcal{C}^1_{\operatorname{SD}}$ to a sequence
	$$
	\ldots\subset \mathcal{C}_{\operatorname{SD}}^{-1} \subset\mathcal{C}^{0}_{\operatorname{SD}}\subset \mathcal{C}_{\operatorname{SD}}^1
	$$
	by defining $\mathcal{C}^{j}_{\operatorname{SD}} \subset \mathcal{C}^0_{\operatorname{SD}}$ as the subset consisting of those $(g_i) \in \mathcal{C}^0_{\operatorname{SD}}$ for which $g_i(\mathcal{G}^{i+j'}) \subset \mathcal{E}^{i+j'}$ for all $0 \leq j' \leq -j$. By assumption (i) $\varphi(\operatorname{Hom}(\mathfrak{P},\mathfrak{M}))$ maps $\mathcal{G}_i$ into $\mathcal{E}_i$ for each $i$. This implies $d_{\operatorname{SD}}$ induces complexes:
	$$
	\mathcal{C}_{\operatorname{SD},j}: \mathcal{C}^{j-1}_{\operatorname{SD}}\rightarrow \mathcal{C}^{j}_{\operatorname{SD}}
	$$
	For $j\leq 0$ we can also define maps $\mathcal{C}^{j}_{\operatorname{SD}} \rightarrow \mathcal{H}^1$ by
	$$
		(h_i) \mapsto (-1)^{-j+1}(\underbrace{0 ,\ldots, 0}_{-j+1},\overline{h}_2,\ldots,\overline{h}_{e+j-1})
	$$
	(here $\overline{h}_i$ denotes the image of $h$ in $\operatorname{Hom}(\mathcal{G}^{i-j+1}/u\mathcal{G}^{i-j},\mathcal{E}^{i-j}/\mathcal{E}^{i-j+1})$). A short computation shows that
	$$
	\begin{tikzcd}
		\mathcal{C}^{j}_{\operatorname{SD}} \ar[r] & \mathcal{H}_1 \\ \mathcal{C}^{j-1}_{\operatorname{SD}} \ar[u]\ar[r] & \mathcal{H}_1 \ar[u]
		\end{tikzcd}
	$$
	commutes for all $j$. If $\mathcal{H}^{j}$ denotes the image of $\mathcal{C}^{j}_{\operatorname{SD}} \rightarrow \mathcal{H}^1$ then the following diagram commutes and has exact rows
	$$
	\begin{tikzcd}
		0 \arrow[r] & \mathcal{C}^{j-1}_{\operatorname{SD}} \arrow{r} & \mathcal{C}^{j}_{\operatorname{SD}} \arrow[r] & \mathcal{H}^{j} \arrow[r] & 0 \\
		0 \arrow[r] & \mathcal{C}^{j-2}_{\operatorname{SD}} \arrow[r] \arrow[u,"d_{\operatorname{SD}}"] &\mathcal{C}^{j-1}_{\operatorname{SD}} \arrow[r] \arrow{u}{d_{\operatorname{SD}}} &  \mathcal{H}^{j-1}  \arrow[u,hook] \arrow[r] & 0
	\end{tikzcd}
	$$
	By considering the associated long exact sequence we deduce that $H^1(\mathcal{C}_{\operatorname{SD},j})$ is finite if and only if $H^1(\mathcal{C}_{\operatorname{SD},j+1})$ is. Since $H^0(\mathcal{C}_{\operatorname{SD},j}) = 0$ if $H^1(\mathcal{C}_{\operatorname{SD},j})$ is $\mathbb{F}$-finite then we also have:
	$$
	\operatorname{dim}_{\mathbb{F}} H^1(\mathcal{C}_{\operatorname{SD,j}}) = \operatorname{dim}_{\mathbb{F}} H^1(\mathcal{C}_{\operatorname{SD},j+1}) + \operatorname{dim}_{\mathbb{F}} \mathcal{H}^j - \operatorname{dim}_{\mathbb{F}} \mathcal{H}^{j+1}
	$$
	It is easy to see $\mathcal{H}^j = 0$ for $j \leq -e+2$. Therefore, since $\mathcal{C}^1_{\operatorname{SD},1} = \mathcal{C}^1_{\operatorname{SD}}$, the result will follow if $H^1(\mathcal{C}_{\operatorname{SD},j}) = 0$ for sufficiently small $j$.
	
	For this, note that if $j \leq -e+2$ then $\mathcal{C}^{j}_{\operatorname{SD}}$ consists of those $(h_i) \in \mathcal{C}^{1}_{\operatorname{SD}}$ with $h_i(\mathcal{G}^{i'}) \subset \mathcal{E}^{i'}$ for all $i' \geq i$. In particular, if $(h_i)$ is such an element then $h_i \in \operatorname{Hom}(\mathfrak{P},\mathfrak{M})$ for each $i$ and so we can choose, by the first paragraph of the proof, $h_1' \in \varphi(\operatorname{Hom}(\mathfrak{P},\mathfrak{M}))$ so that $\varphi^{-1}(h'_1)-h'_1 = h_1+h_2 + \ldots+ h_{e-1}$. Then define 
	$$
	h_2' = h_2 + h_1',~ h_3' = h_3 + h_2',\ldots,~ h_{e-1}' = h_{e-1} + h_{e-2}'
	$$
	Using that $\varphi(h_e')$ maps $\mathcal{G}_i$ into $\mathcal{E}_i$ for every $i$ we deduce $(h_i') \in \mathcal{C}^{j}_{\operatorname{SD}} = \mathcal{C}^{j-1}_{\operatorname{SD}}$ and that $d_{\operatorname{SD}}((h_i')) = (h_i)$. This completes the proof.
\end{proof}

\begin{lemma}\label{largen}
	After replacing $(\mathfrak{M},\mathcal{E})$ with $(u^n\mathfrak{M},u^n\mathcal{E}) \in \operatorname{Mod}_{\mathcal{F}}$ for $n$ sufficiently large conditions (i) and (ii) are satisfied.
\end{lemma}
	\begin{proof}
	If $N \geq 0$ is sufficiently large then $u^N\operatorname{Hom}(\mathcal{G}^1,\mathcal{E}^1)$ is contained in both $u\operatorname{Hom}(\mathfrak{P},\mathfrak{M})$ and $\operatorname{Hom}(\mathcal{G}^i,\mathcal{E}^i)$ for each $i$. For any $n \in \mathbb{Z}$ we have 
	\begin{equation}\label{p-1}
		 F^N(\operatorname{Hom}(\mathfrak{P},u^n\mathfrak{M})) = u^{pn}F^{N-(p-1)n}(\operatorname{Hom}(\mathfrak{P},\mathfrak{M}))
	\end{equation}
	For sufficiently large $n$ we have $F^{N-(p-1)n}(\operatorname{Hom}(\mathfrak{P},\mathfrak{M})) = \varphi(\operatorname{Hom}(\mathfrak{P},\mathfrak{M}))$ and so $u^{pn}F^{N-(p-1)n}(\operatorname{Hom}(\mathfrak{P},\mathfrak{M})) = \varphi(\operatorname{Hom}(\mathfrak{P},u^n\mathfrak{M}))$. This shows $\varphi(\operatorname{Hom}(\mathfrak{P},u^n\mathfrak{M})) \subset u^N\operatorname{Hom}(\mathcal{G}_1,u^n\mathcal{E}_1)$ and so is contained in $u\operatorname{Hom}(\mathfrak{P},u^n\mathfrak{M})$ and $\operatorname{Hom}(\mathcal{G}_i,u^n\mathcal{E}_i)$ for each $i$.
\end{proof}
	\begin{proof}[Proof of Propsition~\ref{extdim}]
	First, note that $H^0(\mathcal{C}_{\operatorname{SD}})$ is contained $\operatorname{Hom}(\mathfrak{P},\mathfrak{M})^{\varphi=1}$  which is always $\mathbb{F}$-finite since it is contained in the finite dimensional $\mathbb{F}$-vector space $(\operatorname{Hom}(\mathfrak{P},\mathfrak{M}) \otimes_{k[[u]]} C^\flat)^{\varphi=1}$. If we replace $(\mathfrak{M},\mathcal{E})$ by $(u^n\mathfrak{M},u^n\mathcal{E}) \in \operatorname{Mod}_{\mathcal{F}}$ in the definition of $\mathcal{C}_{\operatorname{SD}}$ we obtain another complex which we denote $\mathcal{C}_{\operatorname{SD}}(n)$. Taking $n=1$ we obtain an exact sequence
	$$
	0 \rightarrow \mathcal{C}_{\operatorname{SD}}(1) \rightarrow \mathcal{C}_{\operatorname{SD}} \rightarrow \mathcal{Q} \rightarrow 0
	$$
	of complexes, whose associated long exact sequence reads
	\begin{equation}\label{les}
		\begin{aligned}
			0 \rightarrow H^0(\mathcal{C}_{\operatorname{SD}}(1))& \rightarrow H^0(\mathcal{C}_{\operatorname{SD}}) \rightarrow H^0(\mathcal{Q}) \\ &\rightarrow H^1(\mathcal{C}_{\operatorname{SD}}(1)) \rightarrow H^1(\mathcal{C}_{\operatorname{SD}}) \rightarrow H^1(\mathcal{Q}) \rightarrow 0
		\end{aligned}
	\end{equation}
	Note that $\mathcal{Q}$ is a two term complex $\mathcal{Q}^0 \xrightarrow{\gamma} \mathcal{Q}^1$ and the $\mathcal{Q}^i$ can be described explicitly. It is easy to see that $\mathcal{C}_{\operatorname{SD}}(1)^1 = u\mathcal{C}^1_{\operatorname{SD}}$, and so $\mathcal{Q}_1 \cong \mathcal{C}^1_{\operatorname{SD}} / u\mathcal{C}^1_{\operatorname{SD}}$. On the other hand,
	$$
	\mathcal{Q}^0 \cong  \frac{F^0(\operatorname{Hom}(\mathfrak{P},\mathfrak{M}))}{F^0(\operatorname{Hom}(\mathfrak{P},u\mathfrak{M}))}\times \prod_{i=2}^{e-1} \operatorname{Hom}(\mathcal{G}^{i},\mathcal{E}^{i}/u\mathcal{E}_{i}) 
	$$
	We claim
	$$
	\frac{F^0(\operatorname{Hom}(\mathfrak{P},\mathfrak{M}))}{F^0(\operatorname{Hom}(\mathfrak{P},u\mathfrak{M}))} \cong \bigoplus_{i \in p\mathbb{Z}_{\leq 0} \cup \mathbb{Z}_{\geq 1}} \operatorname{gr}^i(\operatorname{Hom}(\mathfrak{P},\mathfrak{M})_k)
	$$
	as $\mathbb{F}$-vector spaces. In particular, we claim both $\mathcal{Q}^0$ and $\mathcal{Q}^1$ are $\mathbb{F}$-finite and so the same is true for the cohomology of $\mathcal{Q}$. Together Lemma~\ref{firstcase} and Lemma~\ref{largen} imply $H^1(\mathcal{C}_{\operatorname{SD}}(n))$ is finite for large $n$. From \eqref{les} we deduce finiteness of $H^1(\mathcal{C}_{\operatorname{SD}})$.  
	
	To verify the claim first choose an $\mathbb{F}$-linear splitting of $0\rightarrow F^1(\operatorname{Hom}(\mathfrak{P},\mathfrak{M})) \rightarrow F^0(\operatorname{Hom}(\mathfrak{P},\mathfrak{M})) \rightarrow \operatorname{gr}^i(\operatorname{Hom}(\mathfrak{P},\mathfrak{M})) \rightarrow 0$ and so write 
	$$
	F^0(\operatorname{Hom}(\mathfrak{P},\mathfrak{M})) \cong F^1(\operatorname{Hom}(\mathfrak{P},\mathfrak{M})) \oplus \operatorname{gr}^0(\operatorname{Hom}(\mathfrak{P},\mathfrak{M}))
	$$
	Note that $F^0(\operatorname{Hom}(\mathfrak{P},u\mathfrak{M}))= u\operatorname{Hom}(\mathfrak{P},\mathfrak{M}) \cap F^1(\operatorname{Hom}(\mathfrak{P},\mathfrak{M}))$, which is the kernel of the surjection $F^1(\operatorname{Hom}(\mathfrak{P},\mathfrak{M})) \rightarrow F^1(\operatorname{Hom}(\mathfrak{P},\mathfrak{M})_k)$. Therefore, 
	$$
	\frac{F^0(\operatorname{Hom}(\mathfrak{P},\mathfrak{M}))}{F^0(\operatorname{Hom}(\mathfrak{P},u\mathfrak{M}))} \cong \operatorname{gr}^0(\operatorname{Hom}(\mathfrak{P},\mathfrak{M})) \oplus F^1(\operatorname{Hom}(\mathfrak{P},\mathfrak{M})_k)
	$$
	Splitting $0 \rightarrow F^{i+1}(\operatorname{Hom}(\mathfrak{P},\mathfrak{M})_k) \rightarrow F^i(\operatorname{Hom}(\mathfrak{P},\mathfrak{M})_k) \rightarrow \operatorname{gr}^i(\operatorname{Hom}(\mathfrak{P},\mathfrak{M})_k) \rightarrow 0$ for $i \geq 1$ allows us to write 
	$$
	F^1(\operatorname{Hom}(\mathfrak{P},\mathfrak{M})_k) \cong \bigoplus_{i \in \mathbb{Z}_{\geq 1}} \operatorname{gr}^i(\operatorname{Hom}(\mathfrak{P},\mathfrak{M})_k)
	$$
	There are also exact sequences $0 \rightarrow \operatorname{gr}^{i-p}(\operatorname{Hom}(\mathfrak{P},\mathfrak{N})) \xrightarrow{u} \operatorname{gr}^i(\operatorname{Hom}(\mathfrak{P},\mathfrak{M})) \rightarrow \operatorname{gr}^i(\operatorname{Hom}(\mathfrak{P},\mathfrak{M})_k) \rightarrow 0$ and, by choosing $\mathbb{F}$-linear splitting, we can identify
	$$
	\operatorname{gr}^0(\operatorname{Hom}(\mathfrak{P},\mathfrak{M})) \cong \bigoplus_{i \in p\mathbb{Z}_{\leq 0}} \operatorname{gr}^i(\operatorname{Hom}(\mathfrak{P},\mathfrak{M})_k)
	$$
	The claim follows. 
	
	To finish the proof note that \eqref{les} implies 
	$$
	\begin{aligned}
		\chi(\mathfrak{P},\mathfrak{M}) &= \chi(\mathfrak{P},u\mathfrak{M}) + \operatorname{dim}_{\mathbb{F}} H^1(\mathcal{C}_{\operatorname{SD}}) - \operatorname{dim}_{\mathbb{F}} H^0(\mathcal{C}_{\operatorname{SD}}) \\
		&= \chi(\mathfrak{P},u\mathfrak{M}) + \operatorname{dim}_{\mathbb{F}} \mathcal{Q}^1 - \operatorname{dim}_{\mathbb{F}} \mathcal{Q}^0
	\end{aligned}
	$$
	Since $\mathcal{C}^1_{\operatorname{SD}}$ is an $\mathbb{F}[[u]]$-lattice inside $\prod_{i=1}^{e-1} \operatorname{Hom}(\mathfrak{P},\mathfrak{M})[\frac{1}{u}]$, $\mathcal{Q}^1$ has $\mathbb{F}$-dimension equal to $(e-1)r$ where $r$ is the $\mathbb{F}[[u]]$-rank of $\operatorname{Hom}(\mathfrak{P},\mathfrak{M})$. The above description of $\mathcal{Q}^0$ shows it has $\mathbb{F}$-dimension $(e-2)r + \sum_{i \not\in p\mathbb{Z}_{\leq 0} \cup \mathbb{Z}_{\geq 1}} \operatorname{dim}_{\mathbb{F}} \operatorname{gr}^i(\operatorname{Hom}(\mathfrak{P},\mathfrak{M})_k)$. Since $r = \sum_i \operatorname{dim}_{\mathbb{F}} \operatorname{gr}^i(\operatorname{Hom}(\mathfrak{P},\mathfrak{M})_k)$ it follows that
	$$
	\chi(\mathfrak{P},\mathfrak{M}) = \chi(\mathfrak{P},u\mathfrak{M}) + \sum_{\substack{i < 0 \\ p\nmid i }} \operatorname{dim}_{\mathbb{F}} \operatorname{gr}^i(\operatorname{Hom}(\mathfrak{P},\mathfrak{M})_k)
	$$
	Using \eqref{p-1} and the assumption that $\operatorname{gr}^i(\operatorname{Hom}(\mathfrak{P},\mathfrak{M})_k) = 0$ for $i <-p$ we deduce
	$$
	\begin{aligned}
		\chi(\mathfrak{P},\mathfrak{M}) &= \chi(\mathfrak{P},u^n\mathfrak{M}) + \sum_{m=0}^{n-1} \left( \sum_{\substack{i < 0 \\ p\nmid i }} \operatorname{dim}_{\mathbb{F}} \operatorname{gr}^{i-(p-1)m}(\operatorname{Hom}(\mathfrak{P},\mathfrak{M})_k)\right) \\
		&= \chi(\mathfrak{P},u^n\mathfrak{M}) +  \sum_{i <0} \operatorname{dim}_{\mathbb{F}} \operatorname{gr}^i(\operatorname{Hom}(\mathfrak{P},\mathfrak{M}))_k \\
		&= \chi(\mathfrak{P},u^n\mathfrak{M}) + \operatorname{dim}_{\mathbb{F}} \frac{\operatorname{Hom}(\mathfrak{P},\mathfrak{M})_k}{F^0(\operatorname{Hom}(\mathfrak{P},\mathfrak{M})_k)}
	\end{aligned}
	$$
	for $n > 2$. Combining this with Lemma~\ref{firstcase} and Lemma~\ref{largen} gives the proposition.
\end{proof}

\subsection{Reformulating the previous result}\label{reform}

We now suppose that $\mathfrak{P}$ and $\mathfrak{M}$ are objects in $\operatorname{Mod}^{\operatorname{SD}}_{\mathcal{F}}$. Attached to $\mathfrak{P}$ is an $e$-tuple of filtered $k\otimes_{\mathbb{F}_p} \mathbb{F}$-modules 
$$
(\overline{\mathfrak{P}}, \overline{\mathcal{G}}^1,\ldots,\overline{\mathcal{G}}^{e-1})
$$
where $\overline{\mathcal{G}}^i = \mathcal{G}^i/u\mathcal{G}^i$ equipped with the one step filtration with $\operatorname{Fil}^1(\overline{\mathcal{G}^i})$ equal the image of $\mathcal{G}^{i+1}$. Thus, we can attach to $\mathfrak{P}$ a Hodge type $\mu(\mathfrak{P})$. Likewise for $\mathfrak{M}$. In this context Proposition~\ref{extdim} can be reformulated as:

\begin{proposition}\label{niceform}
	$\operatorname{dim}_{\mathbb{F}} H^1(\mathcal{C}_{\operatorname{SD}}) = \operatorname{dim}_{\mathbb{F}} H^0(\mathcal{C}_{\operatorname{SD}}) + d(\mu(\mathfrak{P}),\mu(\mathfrak{M}))$
\end{proposition}
\begin{proof}
	As both $\mathfrak{P}$ and $\mathfrak{M}$ are strongly divisible Lemma~\ref{equivalent} implies that the filtered module $\operatorname{Hom}(\mathfrak{P},\mathfrak{M})_k$ identifies with the filtered module $\operatorname{Hom}_{\operatorname{Fil}}(\overline{\mathfrak{P}},\overline{\mathfrak{M}})$ (as described in Section~\ref{filt}). Since the filtrations on $\overline{\mathfrak{P}}$ and $\overline{\mathfrak{M}}$ are concentrated in degrees $[0,p]$ this implies $\operatorname{gr}^i(\operatorname{Hom}(\mathfrak{P},\mathfrak{M})_k)= 0$ for $i<-p$. Thus Proposition~\ref{extdim} applies. Since
	$$
	\operatorname{Hom}(\mathcal{G}^{i+1}/u\mathcal{G}^i,\mathcal{E}^i/\mathcal{E}^{i+1}) = \operatorname{Hom}(\operatorname{Fil}^1(\overline{\mathcal{G}}^i),\operatorname{gr}^0(\overline{\mathcal{E}}^i)) = \frac{\operatorname{Hom}(\overline{\mathcal{G}}^i,\overline{\mathcal{E}}^i)}{\operatorname{Fil}^0(\operatorname{Hom}(\overline{\mathcal{G}^i},\overline{\mathcal{E}}^i))}
	$$
	it follows that $\operatorname{dim}_{\mathbb{F}}\frac{\operatorname{Hom}(\mathfrak{P},\mathfrak{M})_k}{F^0(\operatorname{Hom}(\mathfrak{P},\mathfrak{M})_k)} + \sum_{i=1}^{e-1} \operatorname{dim}_{\mathbb{F}} \operatorname{Hom}(\mathcal{G}^{i+1}/u\mathcal{G}^{i},\mathcal{E}^{i}/\mathcal{E}^{i+1})$ is precisely $d(\mu(\mathfrak{P}),\mu(\mathfrak{M}))$.
\end{proof}
\section{Crystalline vs. strong divisibility}\label{sectionfilonphi}

\subsection{Filtrations on the image of Frobenius}\label{filonphi}
Fix $(\mathfrak{M},\alpha,\mathcal{F})$ corresponding to an $\mathcal{O}$-valued point of $\mathcal{L}^{\mu,\operatorname{conv}}_R$ with $\mathcal{L}^{\mu,\operatorname{conv}}_R$ as in Section~\ref{moduli}. For the moment $\mu$ is any Hodge type concentrated in degrees $[0,h_j]$ with $h_j\geq 0$. Write $V = V_R \otimes_{\alpha} \mathcal{O}$. Then $V[\frac{1}{p}]$ is a crystalline representation of Hodge type $\mu$ and so, as described in Section~\ref{crys} we have  $D:= D_{\operatorname{crys}}(V[\frac{1}{p}])$ a finite free $K_0\otimes_{\mathbb{Q}_p} E$-module and $D_K = D\otimes_{K_0}K$ a filtered $K\otimes_{\mathbb{Q}_p} E$-module of type $\mu$. 

Since $\mathfrak{M}$ and $V$ satisfy \eqref{ident} it follows that $\mathfrak{M}$ is the Breuil--Kisin module associated to $V$ by Kisin as in e.g. \cite[1.2.1]{Kis10} or \cite{Kis06}. In particular there is $\varphi$-equivariant identification
$$
\mathfrak{M}^\varphi \otimes_{\mathfrak{S}} S[\tfrac{1}{p}] \cong D \otimes_{K_0} S[\tfrac{1}{p}]
$$
where $S$ denotes the $p$-adic completion of the divided power envelope of $W(k)[u]$ relative to the ideal generated by $E(u)$, cf. \cite[3.2]{GLS}. Concretely $S$ can be viewed as the subring of $K_0[[u]]$ consisting of series of the form
$$
\sum_{i=0}^\infty a_i \frac{u^i}{e(i)!}, \qquad a_i \in W(k)
$$
where $e(i)$ denotes the largest integer $\leq i/e$.
This allows us to define a $\varphi$-equivariant diffential operator on $\mathfrak{M}^\varphi\otimes_{\mathfrak{S}} S$ by the formula
$$
\mathcal{N}(d\otimes a) = d\otimes (-u\frac{d}{du}(a))
$$
for $d \in D$ and $a \in S$. 

\begin{construction}
	For integers $n_{ij} \geq 0$ inductively define $S\otimes_{\mathbb{Z}_p} E$-submodules $\operatorname{Fil}^{\lbrace n_{ij} \rbrace}$ of $\mathfrak{M}^\varphi \otimes_{\mathfrak{S}} S$ by setting $\operatorname{Fil}^{\lbrace n_{ij}\rbrace} = \mathfrak{M}^\varphi \otimes_{\mathfrak{S}} S$ if every $n_{ij} \leq 0$ and
	$$
	\operatorname{Fil}^{\lbrace n_{ij}\rbrace} = \lbrace x \in \mathfrak{M}^\varphi \otimes_{\mathfrak{S}} S \mid f_{ij}(x) \in \operatorname{Fil}^{n_{ij}}(D_{K,ij}) \text{ for all $ij$ and }\mathcal{N}(x) \in \operatorname{Fil}^{\lbrace n_{ij}-1\rbrace} \rbrace
	$$
	otherwise.
\end{construction}

Tensoring along the map $S \rightarrow K$ sending $u \mapsto \pi$ produces a surjection $f_\pi:\mathfrak{M}^\varphi \otimes_{\mathfrak{S}}S[\frac{1}{p}] \rightarrow D_K$.
Let $f_{ij}$ denote the composition of this surjection with the projection $D_K =\prod_{ij} D_{K,ij} \rightarrow D_{K,ij}$.

\begin{lemma}\label{filproperties}
The $\operatorname{Fil}^{\lbrace n_{ij}\rbrace}$ enjoy the following properties
\begin{enumerate}
	\item $f_{ij}(\operatorname{Fil}^{\lbrace n_{ij} \rbrace})=\operatorname{Fil}^{n_{ij}}(D_{K,ij})$ for each $ij$.
	\item These are filtrations in the sense that $E_{ij}(u)\operatorname{Fil}^{\lbrace n_{ij} - 1_{i'j'}\rbrace} \subset \operatorname{Fil}^{\lbrace n_{ij} \rbrace} \subset \operatorname{Fil}^{\lbrace n_{ij} - 1_{i'j'}\rbrace}$ for every $i'j'$ (here $1_{i'j'}$ denotes the tuple which is zero everywhere but in the $i'j'$-th position where it is $1$).
	\item $E_{i'j'}(u)x \in \operatorname{Fil}^{\lbrace n_{ij}\rbrace}$ implies $x \in \operatorname{Fil}^{\lbrace n_{ij}-1_{i'j'}\rbrace}$.
	\item $\operatorname{Fil}^{\lbrace h_j\rbrace} \cap \mathfrak{M}^\varphi = (\prod E_{ij}(u)^{h_j})\mathfrak{M}$.
\end{enumerate}
\end{lemma}
\begin{proof}
	Properties (1), (2) and (3) follow from \cite[2.1.9]{GLS15}. Part (4) follows from \cite[(2.2.1)]{GLS15}.
\end{proof}

\begin{corollary}\label{intfil}
	$\operatorname{Fil}^{\lbrace n_{ij}\rbrace} \cap \mathfrak{M}^\varphi = \mathfrak{M}^\varphi \cap \left( \prod_{ij} E_{ij}(u)^{n_{ij}} \right) \mathfrak{M}$ and there are exact sequences
	$$
	0 \rightarrow \operatorname{Fil}^{\lbrace n_{ij}-1_{i'j'}\rbrace} \cap \mathfrak{M}^\varphi \xrightarrow{E_{i'j'}(u)} \operatorname{Fil}^{\lbrace n_{ij} \rbrace} \cap \mathfrak{M}^\varphi \xrightarrow{f_{ij}} \operatorname{Fil}^{n_{i'j'}}(D_{K,i'j'}) \cap f_{i'j'}(\mathfrak{M}^\varphi)
	$$
	whose rightmost map becomes surjective after inverting $p$.
	Furthermore,
	$$
	\mathcal{F}^j = \operatorname{Fil}^{\lbrace h_1,\ldots,h_j,0,\ldots,0\rbrace} \cap \mathfrak{M}^\varphi
	$$
	where $(h_1,\ldots,h_j,0,\ldots,0)$ indicates the tuple with $h_{j'}$ in the $ij'$-th position if $j' \leq j$ and $0$ otherwise.
\end{corollary}
\begin{proof}
	Lemma~\ref{filproperties} together with the fact that the kernel of $f_{i'j'}$ on $\mathfrak{M}^\varphi$ is $E_{i'j'}(u)\mathfrak{M}^\varphi$, ensures that the above sequence is exact. If $n_{ij} \geq h_j$ for each $ij$ then $\operatorname{Fil}^{n_{ij}}(D_{K,ij}) =0$ and therefore 
	$$
	\operatorname{Fil}^{\lbrace n_{ij}\rbrace} \cap \mathfrak{M}^\varphi = \left( \prod_{ij} E_{ij}(u)^{n_{ij}-h_j}\right) (\operatorname{Fil}^{\lbrace h_j \rbrace}\cap \mathfrak{M}^\varphi)
	$$
	As a consequence, part (4) of Lemma~\ref{filproperties} gives the first identity when $n_{ij} \geq h_j$. For general $n_{ij}$ we then argue by decreasing induction on $\sum n_{ij}$ (the previous sentence gives the basis case). If $x \in \operatorname{Fil}^{\lbrace n_{ij}\rbrace}\cap \mathfrak{M}^\varphi$ then we can choose $i'j'$ so that the inductive hypothesis gives $E_{i'j'}(u)x \in\operatorname{Fil}^{\lbrace n_{ij}+1_{i'j'}\rbrace}\cap \mathfrak{M}^\varphi = \mathfrak{M}^\varphi \cap \left( \prod_{ij} E_{ij}(u)^{n_{ij}+1_{i'j'}}\right) \mathfrak{M}$. Therefore $x \in \mathfrak{M}^\varphi \cap \left( \prod_{ij} E_{ij}(u)^{n_{ij}}\right) \mathfrak{M}$. A similar argument gives the converse inclusion.
	
	For exactness on the right after inverting $p$ we use that for every $h\geq 0$ each $x \in S[\frac{1}{p}]$ can be written as $x_1 + E(u)^hx_2$ with $x_1 \in \mathfrak{S}[\frac{1}{p}]$ and $x_2 \in S[\frac{1}{p}]$. Since $E(u)^h\mathfrak{M}^\varphi \otimes_{\mathfrak{S}} S \subset \operatorname{Fil}^{\lbrace n_{ij} \rbrace}$ for sufficiently large $h$ it follows that for each $x \in \operatorname{Fil}^{\lbrace n_{ij}\rbrace}$ there exists $x' \in \operatorname{Fil}^{\lbrace n_{ij}\rbrace} \cap \mathfrak{M}^\varphi[\frac{1}{p}]$ such that $f_{ij}(x) = f_{ij}(x')$ for each $ij$. This, combined with (1) of Lemma~\ref{filproperties}, shows that the above sequence is exact on the right after inverting $p$. 
	
	For the final assertion, if we define $\mathcal{F}'^j:=\operatorname{Fil}^{\lbrace h_1,\ldots,h_j,0,\ldots,0\rbrace} \cap \mathfrak{M}^\varphi$ then Lemma~\ref{filproperties} implies $E_j(u)^{h_j} \mathcal{F}'^{j-1} \subset \mathcal{F}'^j \subset \mathcal{F}'^{j-1}$ and $\mathcal{F}'^e = \left(\prod_j E_j(u)^{h_j}\right) \mathfrak{M}$. Also $\mathcal{F}'^j /\mathcal{F}'^{j-1}$ is $\mathcal{O}$-flat. The proof of Proposition~\ref{conv} shows that these properties uniquely determine $\mathcal{F}'^j$ so $\mathcal{F}'^j = \mathcal{F}^j$ as desired.
\end{proof}

\subsection{Integral filtrations}

Our aim is to prove the following:
\begin{proposition}\label{basis}
	There exists a $\mathbb{Z}_p[[u]]$-basis $(e_i)$ of $\mathfrak{M}^\varphi$ and integers $r_i \in [0,p]$ such that
	\begin{enumerate}
		\item $\operatorname{Fil}^{\lbrace p,0,\ldots,0\rbrace} \cap \mathfrak{M}^\varphi$ is generated by $(E_1(u)^{r_i}e_i)$
		\item $e_i = f_i + \pi g_i$ for some $f_i \in \varphi(\mathfrak{M})$ and $g_i \in \mathfrak{M}^\varphi$. 
	\end{enumerate}
\end{proposition}
Here, as in Corollary~\ref{intfil}, $\lbrace p,0,\ldots,0\rbrace$ indicates the tuple with $p$ in the $i1$-th position for each $i$ and $0$ everywhere else. Before proving the proposition we need some preparations. First take any $x= x^{(0)} \in \mathfrak{M}^\varphi \otimes_{\mathfrak{S}} S$ and inductively define 
$$
x^{(i)} = \sum_{l=0}^{i-1} \frac{H(u)^l}{l!}\mathcal{N}^l(x^{(i-1)})
$$
where $H(u) = \frac{E_1(u)}{E_1(0)}$.

\begin{lemma}\label{lift}
	If $f_{i1}(x) \in \operatorname{Fil}^{n}(D_{K,i1})$ for each $i$ and $\delta_{i} = \operatorname{min}\lbrace i,n\rbrace$ then $x^{(i)} \in \operatorname{Fil}^{\lbrace \delta_i,0,\ldots,0\rbrace}$.
\end{lemma}
\begin{proof}
	It suffices to show $\mathcal{N}(x^{(i)}) \in \operatorname{Fil}^{\lbrace \delta_i-1,0,\ldots,0\rbrace}$. Since $\delta_{i-1} \geq \delta_{i} -1$ we may instead show $\mathcal{N}(x^{(i)}) \in \operatorname{Fil}^{\lbrace \delta_{i-1},0,\ldots,0\rbrace}$. Writing $\partial = -u\frac{d}{du}$ we compute
	$$
	\begin{aligned}
		\mathcal{N}(x^{(i)}) &= \sum_{l=0}^{i-1} \left( \frac{H(u)^{l-1}\partial(H(u))}{(l-1)!} + \frac{H(u)^l}{l!}\mathcal{N}^{l+1}(x^{i-1}) \right) \\
		&=\underbrace{\frac{H(u)^{i-1}}{(i-1)!} \mathcal{N}^i(x^{(i-1)})}_{(a)} + \sum_{l=1}^{i-1}\underbrace{(1+ \partial(H(u)))\frac{H(u)^{l-1}}{(l-1)!} \mathcal{N}^l(x^{(i-1)})}_{(b)}
	\end{aligned}
	$$
	If $x \in \operatorname{Fil}^{\lbrace r,0,\ldots,0\rbrace}$ then $H(u)^lx \in \operatorname{Fil}^{\lbrace r+l,0,\ldots,0 \rbrace}$. Therefore $(a)$ is contained in $\operatorname{Fil}^{\lbrace i-1,0,\ldots,0\rbrace} \subset \operatorname{Fil}^{\lbrace\delta_{i-1},0,\ldots,0\rbrace}$. The inductive hypothesis implies $x^{(i-1)} \in \operatorname{Fil}^{\lbrace \delta_{i-1},0,\ldots,0\rbrace}$ and so $\mathcal{N}^l(x^{(i-1)}) \in \operatorname{Fil}^{\lbrace \delta_{i-1} - l,0,\ldots,0\rbrace}$. Since $1+\partial(H(u)) = -H(u)$, each $(b)$ term is contained $\operatorname{Fil}^{\lbrace \delta_{i-1}-l+l,0,\ldots,0\rbrace}= \operatorname{Fil}^{\lbrace \delta_{i-1},0,\ldots,0\rbrace}$ also.
\end{proof}

To apply Lemma~\ref{lift} in the proof of Proposition~\ref{basis} we require some control of the denominators appearing in the operator $\mathcal{N}$. This is given by the following result. It is here that the particular choice of $\pi$ from Section~\ref{notation} when $p=2$ is important.
\begin{theorem}[Gee--Liu--Savitt,Wang]\label{GLS}
	For each $x \in \varphi(\mathfrak{M})$ and $b\geq 1$ there exist $x_i \in \varphi(\mathfrak{M})$ and $a_i \in E$ such that
	$$
	\mathcal{N}^b(x) = \sum_{i=0}^{\infty} E_1(u)^i a_i x_i
	$$
	with $\pi^{p-i} \mid a_i$ for $i=0,\ldots,p-1$.
\end{theorem}
\begin{proof}
	If $\partial = -u\frac{d}{du}$ then  $\partial(E_1(u)) = E_1(0) - E_1(u)$. Therefore
	$$
	\mathcal{N}(E_1(u)^i a_i x_i) = E_1(u)^ia_i(1 - i)\mathcal{N}(x_i) +iE_1(u)^{i-1}E_1(0)a_ix_i
	$$
	Since $E_1(0)$ is divisible by $\pi$ in $W(k)\otimes_{\mathbb{Z}_p}\mathcal{O}$ this shows, by an easy induction, that if the statement holds for $b=1$ then it holds for $b\geq 1$.
	
	If $p>2$ then \cite[4.7]{GLS} implies that $\mathcal{N}(x)$ is contained in $\mathfrak{M}^\varphi \otimes_{\mathfrak{S}} u^pS'$ for $S' = W(k)[[u^p,\tfrac{u^{ep}}{p}]][\tfrac{1}{p}] \cap S$ and $x \in \varphi(\mathfrak{M})$. When $p=2$ the same is true by the results in \cite{Wang17}.
	Therefore, the theorem will follow if every $s \in u^pS'$ can be written as 
	$$\sum_{i=0}^{\infty} E_1(u)^i a_i$$ for $a_i \in W(k) \otimes_{\mathbb{Z}_p} E$ with $\pi^{p-i}\mid a_i$ for $i=0,\ldots,p-1$.  Here we are viewing $s$ as an element of $K_0[[u]]\otimes_{\mathbb{Q}_p} E$ and so also as a tuple $(s_i) \in \prod_i E[[u]]$. Each $s_i$ can be written as 
	$$
	\begin{aligned}
	\sum_{j=0}^\infty a_{ij} \frac{u^{p(j+1)}}{e(pj)!} &= \sum_{j=1}^{\infty} \frac{a_{ij}}{e(pj)!} \left( \sum_{l=0}^{p(j+1)} \binom{p(j+1)}{l}(u-\pi_{i1})^l\pi_{i1}^{p(j+1)-l}\right)\\ &= \sum_{l=0}^\infty (u-\pi_{i1})^l \underbrace{\left( \sum_{j+1 \geq l/p} \frac{a_{ij}}{e(pj)!} \binom{p(j+1)}{l} \pi_{i1}^{p(j+1)-l} \right)}_{a_l}
	\end{aligned}
	$$
	for some $a_{ij} \in \mathcal{O}$ and $\pi_{i1} = \kappa_{i1}(\pi)$. We must show that the $a_l$ term is divisible by $\pi^{p-l}$ in $\mathcal{O}$ for $l=0,\ldots,p-1$. This follows from the observation that $\frac{\pi_{i1}^{pj}}{e(pj)!} \in \mathcal{O}$.
\end{proof}

\begin{corollary}
	Suppose that $x\in \varphi(\mathfrak{M})$. Then for each $i \leq p$ there are $x_1,\ldots,x_{p-1} \in \mathfrak{M}^\varphi$ 
	$$
	x^{(i)} -x + \pi^{p}x_1 + E_{1}(u)\pi^{p-1}x_2+\ldots + E_{1}(u)^{p-1} \pi x_{p-1}  \in E_{1}(u)^p\mathfrak{M}^\varphi \otimes_{\mathfrak{S}} S[\tfrac{1}{p}]
	$$

\end{corollary}
\begin{proof}
	Using Theorem~\ref{GLS} it suffices to show that $x^{(i)}-x$ can be written as a $\mathbb{Z}$-linear combination of terms $\frac{H(u)^a}{a'!}\mathcal{N}^b(x)$ for $a,b\geq 1$ and $1\leq a'\leq a$. Arguing by induction it is enough to show
	$$
	\frac{H(u)^l}{l!}\mathcal{N}^l(\frac{H(u)^a}{a'}\mathcal{N}^b(x)) = \sum_{k=0}^l \binom{l}{k} \frac{H(u)^l \partial^k(H(u)^a)}{l!a'!} \mathcal{N}^{l-k+b}(x)
	$$
	is a $\mathbb{Z}$-linear combination of such terms. This follows from the claim that $\frac{\partial^k(H(u)^a)}{a!}$ is a $\mathbb{Z}$-linear combination of terms of the form $\frac{H(u)^{a'}}{a'!}$ for $1 \leq a' \leq a$. To see this note that $\partial(H(u)^a) = aH(u)^{a-1}(-1-H(u))$ and so $\partial^k(H(u)^a)/a!$ equals
	$$
	\begin{aligned}
		\frac{1}{(a-1)!} \partial^{k-1}(H(u)(-1-H(u))) &= \frac{1}{(a-1)!} \sum_{j=0}^{k-1} \binom{k-1}{j} \partial^j(H(u)^{a-1}) \partial^{k-1-j}(-1-H(u)) \\
		&= \frac{1}{(a-1)!}(-1-H(u)) \sum_{j=0}^{k-1} \binom{k-1}{j} (-1)^{k-1-j}\partial^j(H(u)^{a-1})
	\end{aligned}
	$$
	(for the second equality we've used that $\partial^n(H(u)) = (-1)^n(-1-H(u))$ for $n >0$). Inducting on $k$ finishes the proof.
\end{proof}
\begin{proof}[Proof of Proposition~\ref{basis}]
	Set $M_j$ equal to the image of $\mathfrak{M}^\varphi$ under $f_j := \prod_i f_{ij}$. This is a $W(k)\otimes_{\mathbb{Z}_p} \mathcal{O}$-lattice inside $\prod_i D_{K,ij}$ which we equip with the filtration 
	$$
	\operatorname{Fil}^{n}(M_{j}) = M_{j} \cap \prod_i  \operatorname{Fil}^n(D_{K,ij})
	$$
	Choose a $\mathbb{Z}_p$-basis $(\overline{e}_i)$ of $M_1$ adapted to the filtration, i.e. so that there exists integers $r_i$ with $\operatorname{Fil}^n(M)$ is generated by those $\overline{e}_i$ with $n\leq r_i$. This is possible since the graded pieces of the filtration on $M_1$ are $p$-torsionfree by construction.
	
	Note that $f_{1}$ induces a surjection $\varphi(\mathfrak{M}) \rightarrow \mathfrak{M}^\varphi \rightarrow  M_1$ and so we can lift $(\overline{e}_i)$ to a $\mathbb{Z}_p[[u^p]]$-basis $(f_i)$ of $\varphi(\mathfrak{M})$. Lemma~\ref{lift} implies 
	$$
	f_i^{(p)} \in \operatorname{Fil}^{\lbrace \operatorname{min}\lbrace r_i,p\rbrace,0,\ldots,0\rbrace}
	$$
	If $f_i^{(p)} = f_i + \pi f_i' + H(u)^pf_i''$ with $f_i' = \pi^{p-1} f_{i,1} + \ldots + E_{ij}(u)^{p-1} f_{i,p-1}$ as in Lemma~\ref{lift} then 
	$$
	e_i := f_i +\pi f_i' \in \operatorname{Fil}^{\lbrace \operatorname{min}\lbrace r_i,p\rbrace,0,\ldots,0\rbrace} \cap \mathfrak{M}^\varphi
	$$
	To finish the proof we show by induction on $n$ that $\operatorname{Fil}^{\lbrace n,0,\ldots,0\rbrace} \cap \mathfrak{M}^\varphi$ is equal to the submodule $Y_n$ generated over $\mathbb{Z}_p[[u]]$ by $E_1(u)^{\operatorname{max}\lbrace n-r_i,0\rbrace}e_i$ whenever $n \leq p$. The case $n=0$ is clear so assume that assertion holds for $n-1$. Since the image of $e_i$ in $M$ equals that of $f_i$, the image of $Y_n$ in $M$ equals $\operatorname{Fil}^n(M)$ and so contains the image of $\operatorname{Fil}^{\lbrace n,0,\ldots,0\rbrace}\cap \mathfrak{M}^\varphi$. Corollary~\ref{intfil} shows that the kernel of $\operatorname{Fil}^{\lbrace n,0,\ldots,0\rbrace}\cap \mathfrak{M}^\varphi \rightarrow M$ equals $E_1(u)\operatorname{Fil}^{\lbrace n-1,0,\ldots,0\rbrace} \cap \mathfrak{M}^\varphi$ which, by the inductive hypothesis, equals $E_1(u)Y_{n-1}$. Since $E_1(u)Y_{n-1} \subset Y_n$ we conclude that $Y_n = \operatorname{Fil}^{\lbrace n,0,\ldots,0\rbrace}\cap \mathfrak{M}^\varphi$ as desired.
\end{proof}

\subsection{Application to strong divisibility}

Maintain the notation from above but assume additionally that $\mu$ is pseudo-Barsotti--Tate, i.e. that $h_1=p$ and $h_2=\ldots=h_e=1$.

\begin{proposition}\label{hodgetypes}
	If $(\mathfrak{M},\alpha,\mathcal{F})$ corresponds to an $A$-valued point of $\mathcal{L}^{\mu,\operatorname{conv}}_R$ for any finite flat $\mathcal{O}$-algebra $A$ then
	$$
	(\mathfrak{M}_{\mathbb{F}},\mathcal{F}_{\mathbb{F}}) := (\mathfrak{M},\mathcal{F})\otimes_{\mathcal{O}} \mathbb{F}
	$$
	is an object of $\operatorname{Mod}^{\operatorname{SD}}_{\mathcal{F}}$. If $A = \mathcal{O}$ then this object has Hodge type $\mu$ (in the sense of Section~\ref{reform}) .
\end{proposition}
\begin{proof}
	The first part follows immediately by viewing $\mathfrak{M}$ as an $\mathfrak{S}_{\mathcal{O}}$-module rather than an $\mathfrak{S}_A$-module and $V_\alpha$ as an $\mathcal{O}$-representation rather than a representation on an $A$-module, and then applying Proposition~\ref{basis}.
	
	For the second part, from $(\mathfrak{M},\mathcal{F})$ we obtain filtered $W(k) \otimes_{\mathbb{Z}_p} \mathcal{O}$-modules for $j=1,\ldots,e$
	$$
	\overline{\mathcal{F}}^{j-1} = \mathcal{F}^{j-1} / E_j(u),\qquad  \operatorname{Fil}^n(\overline{\mathcal{F}}^{j-1}) = \text{the image of $E_j(u)^n\mathcal{F}^j \cap \mathcal{F}^{j-1}$}
	$$
	Corollary~\ref{intfil} implies that $f_j := \prod_i f_{ij}$ induces an embedding $\overline{\mathcal{F}}^{j-1} \hookrightarrow \prod_i D_{K,ij}$ of filtered modules, which becomes an isomorphism after inverting $p$. Therefore the $e$-tuple of filtered $K_0 \otimes_{\mathbb{Q}_p} E$-modules $\overline{\mathcal{F}}^{j-1}[\frac{1}{p}]$ has Hodge type $\mu$ (in the sense of Section~\ref{Hodgetype}). 
	
	In general the graded pieces of the $\overline{\mathcal{F}}^{j-1}$ may not be $\mathcal{O}$-flat. However, we will show that this is the case when $\mu$ is pseudo-Barsotti--Tate. This will imply that the $e$-tuple $\overline{\mathcal{F}}^{j-1} \otimes_{\mathcal{O}} \mathbb{F}$ of filtered $k \otimes_{\mathbb{F}_p} \mathbb{F}$-modules also has Hodge type $\mu$. Firstly, consider for $j>1$. Then
	$$
	\operatorname{Fil}^n(\overline{\mathcal{F}}^{j-1}) = \begin{cases}
		0 & \text{for $n\geq 1$}\\ 
		\text{the image of $\mathcal{F}^j$} & \text{for $n=0$} \\
		\overline{\mathcal{F}}^{j-1} & \text{for $n<0$} 
	\end{cases}
	$$
 so the graded pieces are clearly $\mathcal{O}$-flat. This also shows that $\overline{\mathcal{F}}^{j-1} \otimes_{\mathcal{O}} \mathbb{F} = \overline{\mathcal{F}}^{j-1}_{\mathbb{F}}$ where $\overline{\mathcal{F}}^{j-1}_{\mathbb{F}}$ is the filtration attached to $(\mathfrak{M}_{\mathbb{F}},\mathcal{F}_{\mathbb{F}})$ as in Section~\ref{reform}. We claim that the same assertions hold when $j=1$. For this we use the notation from Proposition~\ref{reform}: after recalling that $\operatorname{Fil}^{\lbrace p,0,\ldots,0\rbrace} \cap \mathfrak{M}^\varphi = \mathcal{F}^1$ we have $\mathfrak{M}^\varphi \cap E_1(u)^{n}\mathcal{F}^1$ generated over $\mathbb{Z}_p[[u]]$ by those $E(u)^{\operatorname{max}\lbrace 0,r_i +n\rbrace}e_i$. Therefore $\operatorname{Fil}^n(\overline{\mathcal{F}}^0)$ is generated over $\mathbb{Z}_p$ by the images of those $e_i$'s with $r_i + n \leq n$. This shows the graded pieces are $\mathcal{O}$-flat. It also shows that the map $\overline{\mathcal{F}}^0 \otimes_{\mathcal{O}} \mathbb{F} \rightarrow \overline{\mathfrak{M}}_{\mathbb{F}}$ (where again $\overline{\mathfrak{M}}_{\mathbb{F}}$ is equipped with the filtration as in Section~\ref{reform}) is an isomorphism of filtered $k\otimes_{\mathbb{F}_p} \mathbb{F}$-modules since an identical description of the filtered pieces on  $\overline{\mathfrak{M}}_{\mathbb{F}}$ can be given in terms of the images of the $e_i$'s in $\mathfrak{M}_{\mathbb{F}}$. This finishes the proof.
\end{proof}
\section{Tangent spaces}

\subsection{Tangent space dimensions and extension groups}\label{tgt}

For any $\mathbb{F}$-valued point $x \in \mathcal{L}^{\mu,\operatorname{conv}}_R$ corresponding to $(\mathfrak{M}_x,\alpha_x,\mathcal{F}_x)$ we write $T_x$ for the tangent space of $\mathcal{L}^{\mu,\operatorname{conv}}_R \otimes_{\mathcal{O}} \mathbb{F}$ at $x$. In other words $T_x = \mathcal{L}^{\mu,\operatorname{conv}}_R(\mathbb{F}[\epsilon])$ where $\mathbb{F}[\epsilon]$ denotes the ring of dual numbers over $\mathbb{F}$. To understand these vector spaces observe that if $(\mathfrak{M},\alpha,\mathcal{F}) \in T_x$ then, since $\mathfrak{M}$ and $\mathcal{F}^i$ are $\mathbb{F}[\epsilon]$-flat,
$0\rightarrow \mathfrak{M}_x\rightarrow \mathfrak{M} \xrightarrow{\epsilon} \mathfrak{M}_x\rightarrow 0
$
is an exact sequence in $\operatorname{Mod}_{\mathcal{F}}$ (here we write $\mathfrak{M}_x$ in place of $(\mathfrak{M}_x,\mathcal{F}_x)$ and likewise for $\mathfrak{M}$). This construction produces an $\mathbb{F}$-linear homomorphism
$$
T_x \rightarrow \operatorname{Ext}^1_{\mathcal{F}}(\mathfrak{M}_x,\mathfrak{M}_x)
$$
into the Yoneda extension group in $\operatorname{Mod}_{\mathcal{F}}$.

\begin{proposition}\label{sdextensions}
	If $\mu$ is pseudo-Barsotti--Tate then $T_x\rightarrow \operatorname{Ext}^1_{\mathcal{F}}(\mathfrak{M}_x,\mathfrak{M}_x)$ factors through $\operatorname{Ext}^1_{\operatorname{SD}}(\mathfrak{M}_x,\mathfrak{M}_x)$.
\end{proposition}
\begin{proof}
	Since $\mathcal{L}^{\mu,\operatorname{conv}}_R$ is $\mathcal{O}$-flat and $\mathcal{L}^{\mu,\operatorname{conv}}_R[\frac{1}{p}] = \operatorname{Spec}R^\mu[\frac{1}{p}]$ which is reduced it follows from \cite[4.1.2]{B19} that every $\overline{A}$-valued point of $\mathcal{L}^{\mu,\operatorname{conv}}_R$ valued in an Artin local ring with finite residue field is induced from an $A$-valued point for $A$ some finite flat $\mathcal{O}$-algebra. The claim therefore follows by applying this with $\overline{A} = \mathbb{F}[\epsilon]$ and using the first part of Proposition~\ref{hodgetypes}.
\end{proof}

\subsection{Cyclotomic-freeness}
To describe the kernel of $T_x \rightarrow \operatorname{Ext}^1_{\mathcal{F}}(\mathfrak{M}_x,\mathfrak{M}_x)$ we will need to use the cyclotomic-freeness assumption. We will also need a second technical result from \cite{GLS}. It is very closely related to Theorem~\ref{GLS} (in fact it is the main ingredient going into the proof of Theorem~\ref{GLS}). 

\begin{theorem}[Gee--Liu--Savitt]
	Suppose that $A$ is a finite local $\mathcal{O}$-algebra and $(\mathfrak{M},\alpha,\mathcal{F}) \in \mathcal{L}^{\mu,\operatorname{conv}}_R$ for any $\mu$. Then the identification
	$$
	\mathfrak{M} \otimes_{\mathfrak{S}} W(C^\flat) = V_\alpha \otimes_{\mathbb{Z}_p} W(C^\flat)
	$$
	is such that $(\sigma-1)(m) \in \mathfrak{M} \otimes [\pi^\flat]\varphi^{-1}\mu A_{\operatorname{inf}}$ for all $\sigma \in G_K$ and all $m \in \mathfrak{M}$. Here $\mu = [\epsilon]-1$ for some generator $\epsilon \in \mathbb{Z}_p(1)$. 
\end{theorem}
\begin{proof}
	We reduce to the case where $A$ is $\mathcal{O}$-flat using \cite[4.1.2]{B19}. In this case the Theorem is one direction of \cite[2.1.12]{B19}.
\end{proof}

When $pA=0$ we have $\mathfrak{M} \otimes_{\mathfrak{S}} [\pi^\flat]\varphi^{-1}(\mu)A_{\operatorname{inf}} = \mathfrak{M} \otimes_{k[[u]]} u^{(e+p-1)/(p-1)}\mathcal{O}_{C^\flat}$ as follows from the well-known calculation that $\epsilon-1$ generates the ideal $u^{ep/(p-1)}\mathcal{O}_{C^\flat}$ whenever $\epsilon$ is a compatible system of primitive $p$-th power roots of unity, cf. \cite[5.2.1]{Fon00}. This motivates the following setup. Let $\mathfrak{M}_1,\mathfrak{M}_2$ be Breuil--Kisin modules over $\mathfrak{S}_{\mathbb{F}}$ satisfying
\begin{enumerate}
	\item $u^{e+p-1}\mathfrak{M}_i \subset \mathfrak{M}^\varphi_i$.
	\item Each $\mathfrak{M}_i \otimes_{k[[u]]} C^\flat$ is equipped with a $\varphi$-equivariant $C^\flat$-semilinear action of $G_K$ for which
	\begin{enumerate}
		\item $(\sigma-1)(m) \in \mathfrak{M}_i \otimes_{k[[u]]} u^{(e+p-1)/(p-1)}\mathcal{O}_{C^\flat}$ for $\sigma \in G_K$
		\item $(\sigma-1)(m) = 0$ for $\sigma \in G_{K_\infty}$
	\end{enumerate}
whenever $m \in \mathfrak{M}_i$.
\end{enumerate}

\begin{proposition}\label{unique}
		Let $\gamma \in \overline{k}$ be such that $\sigma(\gamma u^{(p+e-1)/(p-1)}) = \gamma u^{(p+e-1)/(p-1)} \chi_{\operatorname{cyc}}(\sigma)$ for all $\sigma \in G_{K_\infty}$. 
		If there exists no $\varphi$-equivariant $\mathfrak{S}_A$-linear map
		$$
		\mathfrak{M}_1 \rightarrow \gamma u^{(p+e-1)/(p-1)}\mathfrak{M}_2
		$$
		then any $\varphi$-equivariant $\mathfrak{S}_{\mathbb{F}}$-linear map $\mathfrak{M}_1 \rightarrow \mathfrak{M}_2$ becomes  $G_K$-equivariant after extending scalars to $\mathcal{O}_{C^\flat}$.
	\end{proposition}
We point out that such a $\gamma$ exists because the character of $G_{K_\infty}$ defined by $\sigma(u^{(e+p-1)/(p-1)}) = \chi(\sigma)u^{(e+p-1)/(p-1)}$ is an unramified twist of the $\chi_{\operatorname{cyc}}$.
	\begin{proof}
		The $G_K$-action on $\operatorname{Hom}(\mathfrak{M}_1,\mathfrak{M}_2) \otimes_{k[[u]]} C^\flat$ given by $h \mapsto \sigma \circ h \circ \sigma^{-1}$ is $\varphi$-equivariant. We must show $(\sigma-1)(H) =0$ for all $\sigma \in G_K$ if $H \in \operatorname{Hom}(\mathfrak{M},\mathfrak{N})$ satisfies $(\varphi-1)(H)=0$. Assumption (2a) implies $(\sigma-1)(H) \in \mathcal{H} \otimes_{k[[u]]} \mathcal{O}_{C^\flat}$ for
		$$
		\mathcal{H} = \operatorname{Hom}(\mathfrak{M}_1,\gamma u^{(p+e-1)/(p-1)}\mathfrak{M}_2)
		$$
		Since the element $\gamma$ satisfies $\varphi(\gamma)/\gamma \in k$ assumption (1) implies
		$$
		(\gamma u^{(p+e-1)/(p-1)} \mathfrak{N})^\varphi \subset \gamma u^{p+e-1+(p+e-1)/(p-1)}\mathfrak{N}
		$$
		It follows that $\mathcal{H}$ is $\varphi$-stable inside $\operatorname{Hom}(\mathfrak{M},\mathfrak{N}) \otimes_k \overline{k}$. Since $H$ is $\varphi$-equivariant $\sigma \mapsto (\sigma-1)(H)$ defines a continuous $1$-cocycle
		$$
		G_K \rightarrow (\mathcal{H}\otimes_{k[[u]]} \mathcal{O}_{C^\flat})^{\varphi=1}
		$$
		which vanishes on $G_{K_\infty}$. We will show any such cocycle is zero. First, since $\mathcal{H}$ is $\varphi$-stable it is easy to see that $V:= (\mathcal{H}\otimes_k \overline{k})^{\varphi=1} = (\mathcal{H} \otimes_{k[[u]]} \mathcal{O}_{C^\flat})^{\varphi=1}$. By the choice of $\gamma$ and the fact that $G_{K_\infty}$-acts as the identity on $\operatorname{Hom}(\mathfrak{M}_1,\mathfrak{M}_2)$ it follows that the action of $G_{K_\infty}$ on $V \otimes_{\mathbb{F}} \mathbb{F}(-1)$ is unramified (i.e, is trivial on $G_{K_\infty} \cap I_K$ where $I_K \subset G_K$ denotes the inertia subgroup). We also claim that $V \otimes_{\mathbb{F}} \mathbb{F}(-1)$ contains no element invariant under $G_{K_\infty}$. This follows because, $K_\infty$ being totally ramified, the composite $G_{K_\infty} \rightarrow G_K \rightarrow G_k$ is surjective and so any such element would lie in $\mathcal{H}^{\varphi=1}$. However, by assumption no such element exists. The proposition then follows from Lemma~\ref{cohom} below.
	\end{proof}
	\begin{lemma}\label{cohom}
		Let $V$ be a continuous representation of $G_K$ on an $\mathbb{F}$-vector space such that $V\otimes_{\mathbb{F}} \mathbb{F}(-1)|_{G_{K_\infty}}$ is unramified and contains no $G_{K_\infty}$-invariant element. Then any continuous $1$-cocycle $G_K \rightarrow V$ which vanishes on $G_{K_\infty}$ is zero.
	\end{lemma}
	\begin{proof}
		Let $\widehat{K} = K_\infty(\mu_{p^\infty})$. Since the cyclotomic character is trivial on $G_{\widehat{K}}$ our first assumption implies that $G_{\widehat{K}}$ acts on $V$ through the composite $G_{\widehat{K}} \hookrightarrow G_K \rightarrow G_k$. This composite is surjective and so our second assumption implies $V^{G_{\widehat{K}}}=0$. Inflation-restriction therefore implies $H^1(G_K,V)\rightarrow H^1(G_{\widehat{K}},V)$ is injective and so $H^1(G_K,V)\rightarrow H^1(G_{K_\infty},V)$ is also. It follows that any $1$-cocycle as in the lemma is a coboundary $\sigma \mapsto (\sigma-1)(v)$ for some $v\in V^{G_{K_\infty}}$. However $V^{G_{K_\infty}} \subset V^{G_{\widehat{K}}}$ which we've just seen is zero.
	\end{proof}
	
	\begin{lemma}\label{cyclofree}
		Suppose that for every unramified $G_{K_\infty}$-submodule $V \subset V_{\mathbb{F}}$ there exists no $G_{K_\infty}$-equivariant quotient $V_{\mathbb{F}} \rightarrow W$ with $V \cong W\otimes \mathbb{F}(1)$. Then there exists no non-zero $\varphi$-equivariant map
		$$
		\mathfrak{M}_x \rightarrow \gamma u^{(e+p-1)/(p-1)}\mathfrak{M}_x
		$$
	\end{lemma}
\begin{proof}
	First consider $\mathfrak{M}$ and $\mathfrak{N}$ with $u^{e+p-1}\mathfrak{M} \subset \mathfrak{M}^\varphi$ and $\mathfrak{N}^\varphi \subset u^{e+p-1}\mathfrak{N}$ and suppose $H:\mathfrak{M}\rightarrow \mathfrak{N}$ is non-zero and $\varphi$-equivariant. Then there is a non-zero induced $G_{K_\infty}$-equivariant map
	\begin{equation}\label{map}
	(\mathfrak{M}\otimes_{k[[u]]} C^\flat)^{\varphi=1} \rightarrow (\mathfrak{N} \otimes_{k[[u]]} C^\flat)^{\varphi=1}
\end{equation}
	We claim that the action of $G_{K_\infty}$ on the image of this map is unramified after twisting by $\mathbb{F}(1)$. Applying this with $\mathfrak{M} = \mathfrak{M}_x$ and $\mathfrak{N} = \gamma u^{(e+p-1)/(p-1)}\mathfrak{M}_x$ gives the lemma since then $(\mathfrak{N} \otimes_{k[[u]]} C^\flat)^{\varphi=1}= V_{\mathbb{F}} \otimes \mathbb{F}(-1)$ and $(\mathfrak{M}\otimes_{k[[u]]} C^\flat)^{\varphi=1} = V_{\mathbb{F}}$.
	
	To establish the claim we can assume that $H$ is injective and becomes surjective after inverting $u$ so that \eqref{map} is an isomorphism. Choose bases of $\mathfrak{M}$ and $\mathfrak{N}$ so that their Frobenii are respectively represented by matrices $A$ and $B$, and so that $H$ is represented by $C$. The $\varphi$-equivariance of $H$ implies $B\varphi(C)A^{-1} = C$. The fact that $u^{e+p-1}\mathfrak{M} \subset \mathfrak{M}^\varphi$ and $\mathfrak{N}^\varphi \subset u^{e+p-1}\mathfrak{N}$ implies $u^{-(e+p-1)}B$ and $u^{e+p-1}A^{-1}$ have coefficients in $\mathfrak{S}_{\mathbb{F}}$. Considering the $u$-adic valuation of determinants in the identity $u^{-(e+p-1)}B \varphi(C) u^{e+p-1}A^{-1} = C$ implies that $C, u^{-(e+p-1)}B$ and $u^{(e+p-1)}A$ are invertible over $\mathfrak{S}_{\mathbb{F}}$. In particular, the Frobenius on $\mathfrak{M}' := (\gamma u^{(e+p-1)/(p-1)})^{-1}\mathfrak{M}$ is a semilinear automorphism and so the $G_{K_\infty}$-action on $(\mathfrak{M}' \otimes_{k[[u]]} C^\flat)^{\varphi=1}$ is unramified. The definition of $\gamma$ gives that $(\mathfrak{M}' \otimes_{k[[u]]} C^\flat)^{\varphi = 1} = (\mathfrak{M} \otimes_{k[[u]]} C^\flat)^{\varphi=1} \otimes \mathbb{F}(1)$ as $G_{K_\infty}$-representations, and the claim follows.
\end{proof}

\begin{lemma}\label{cyclofree=>}
	The conclusions of Lemma~\ref{cyclofree} apply if $V_{\mathbb{F}}$ is cyclotomic-free.
\end{lemma}
\begin{proof}
	Restriction from $G_K$ to $G_{K_\infty}$ is an equivalence between irreducible representations of either group (cf. \cite[2.2.1]{B18b}) which respects being unramified. Thus, if $V$ and $W$ as in Lemma~\ref{cyclofree} exist then there is an unramified $G_K$-Jordan--Holder factor $V'$ of $V_{\mathbb{F}}$ for which $V' \otimes \mathbb{F}(1)$ is also a $G_K$-Jordan--Holder factor. Thus $V_{\mathbb{F}}$ is not cyclotomic-free.
\end{proof}
\subsection{Tangent space bounds}

\begin{proposition}\label{bound}
	If $\mu$ is pseudo-Barsotti--Tate and there are no  non-zero $\varphi$-equivariant maps
	$$
	\mathfrak{M}_x \rightarrow \gamma u^{(e+p-1)/(p-1)}\mathfrak{M}_x
	$$
	then $\operatorname{dim}_{\mathbb{F}}T_x \leq d^2 + d(\mu,\mu)$.
\end{proposition}
\begin{proof}
	By Proposition~\ref{sdextensions}, Proposition~\ref{niceform}, and Proposition~\ref{hodgetypes} it suffices to show that the kernel of $T_x \rightarrow \operatorname{Ext}^1_{\operatorname{SD}}(\mathfrak{M}_x,\mathfrak{M}_x)$ has $\mathbb{F}$-dimension 
	$$
	\leq d^2 - \operatorname{Hom}_{\mathcal{F}}(\mathfrak{M}_x,\mathfrak{M}_x)
	$$
	For this suppose $(\mathfrak{M}_i,\alpha_i,\mathcal{F}_i)$ are in this kernel for $i=1,2$. Observe that there is no non-zero $\varphi$-equivariant map $H :\mathfrak{M}_1 \rightarrow \gamma u^{(e+p-1)/(p-1)}\mathfrak{M}_2$
	as in Proposition~\ref{unique}. indeed, since $\mathfrak{M}_1$ and $\mathfrak{M}_2$ are both extensions of $\mathfrak{M}_x$ by itself any non-zero such $H$ would induce a non-zero $\mathfrak{M}_x \rightarrow \gamma u^{(e+p-1)/(p-1)}\mathfrak{M}_x$.
	
	The fact that $(\mathfrak{M}_i,\alpha_i,\mathcal{F}_i)$ are both in the kernel of $T_x \rightarrow \operatorname{Ext}^1_{\operatorname{SD}}(\mathfrak{M}_x,\mathfrak{M}_x)$ implies the existence of a morphism $\alpha:\mathfrak{M}_1 \rightarrow \mathfrak{M}_2$ in $\operatorname{Mod}_{\mathcal{F}}$ which is the identity on $\mathfrak{M}_x$ viewed as either a submodule and a quotient. The previous paragraph combined with Proposition~\ref{unique} shows this identifies with a $G_K$-equivariant map $V_{\alpha_1} \rightarrow V_{\alpha_2}$ after base-changing to $W(C^\flat)$ which induces the identity on $V_{\mathbb{F}}$ when viewed as either a submodule or a quotient.
	
	In particular, it follows that the kernel of $T_x \rightarrow \operatorname{Ext}^1_{\operatorname{SD}}(\mathfrak{M}_x,\mathfrak{M}_x)$ is contained in the kernel of the composite $T_x \rightarrow T \rightarrow \operatorname{Ext}^1(V_{\mathbb{F}},V_{\mathbb{F}})$ where $T$ denotes the tangent space of $R\otimes_{\mathcal{O}} \mathbb{F}$ at its closed point. Since the kernel of $T\rightarrow \operatorname{Ext}^1(V_{\mathbb{F}},V_{\mathbb{F}})$ identifies with $$
	\operatorname{Hom}(V_{\mathbb{F}},V_{\mathbb{F}})/\operatorname{Hom}(V_{\mathbb{F}},V_{\mathbb{F}})^{G_K}
	$$
	we are reduced to considering the kernel of $T_x\rightarrow T$. The group of $G_K$-equivariant automorphisms of $V_{\mathbb{F}} \oplus \epsilon V_{\mathbb{F}}$ which are the identity on $\epsilon V_{\mathbb{F}}$ and modulo $\epsilon$ act on this kernel. This group identifies with $\operatorname{Hom}(V_{\mathbb{F}},V_{\mathbb{F}})$ via $h \mapsto a+b\epsilon \mapsto a + h(b)\epsilon$. The first paragraph shows that this action is transitive. Furthermore, the stabiliser of the zero element in $T_x$ clearly identifies with those $h \in \operatorname{Hom}(V_{\mathbb{F}},V_{\mathbb{F}})^{G_K}$ inducing an endomorphism of $\mathfrak{M}_x$ which is a morphism in $\operatorname{Mod}_{\mathcal{F}}$. In this way we identify the kernel of $T_x\rightarrow T$ with 
	$$\operatorname{Hom}(V_{\mathbb{F}},V_{\mathbb{F}})^{G_K}/\operatorname{Hom}_{\mathcal{F}}(\mathfrak{M}_x,\mathfrak{M}_x)
	$$
	(where $\operatorname{Hom}_{\mathcal{F}}(\mathfrak{M}_x,\mathfrak{M}_x))$ is viewed as a submodule of $\operatorname{Hom}(V_{\mathbb{F}},V_{\mathbb{F}})^{G_K}$ using Proposition~\ref{unique}). This gives the desired bound.
\end{proof}

\begin{corollary}\label{surjective}
	If $\mu$ is pseudo-Barsotti--Tate and $V_{\mathbb{F}}$ is cyclotomic-free then the map $T_x \rightarrow \operatorname{Ext}^1_{\operatorname{SD}}(\mathfrak{M}_x,\mathfrak{M}_x)$ from Proposition~\ref{sdextensions} is surjective.
\end{corollary}
\begin{proof}
	If $T_x \rightarrow \operatorname{Ext}^1_{\operatorname{SD}}(\mathfrak{M}_x,\mathfrak{M}_x)$ is not surjective then the bound in Proposition~\ref{bound} would be strict. However, the discussion in the proof of Theorem~\ref{main} illustrates that  $\operatorname{dim}_{\mathbb{F}}T_x \geq d^2 + d(\mu,\mu)$.
\end{proof}
\section{Applications}

\subsection{Constructing crystalline lifts}

Given two Hodge types $\mu$ and $\mu'$ we write $(\mu,\mu')$ for the Hodge type obtained by taking $\mu_{ij} \cup \mu_{ij}'$ for each $ij$. Note that if $V$ and $V'$ are crystalline representations of Hodge type $\mu$ and $\mu'$ respectively then $V\oplus V'$ has Hodge type $(\mu,\mu')$.
\begin{lemma}\label{liftingextensions}
	Suppose that $(\mathfrak{M},\alpha,\mathcal{F})$ and $(\mathfrak{M}',\alpha',\mathcal{F}')$ respectively correspond to $\mathcal{O}$-valued points in $\mathcal{L}^{\mu,\operatorname{conv}}_{R_{V_{\mathbb{F}}}}$ and $\mathcal{L}^{\mu',\operatorname{conv}}_{R_{V_{\mathbb{F}}'}}$ and consider an exact sequence
	\begin{equation}\label{sequencetolift}
	0 \rightarrow (\mathfrak{M},\mathcal{F}) \otimes_{\mathcal{O}} \mathbb{F}\rightarrow (\mathfrak{N}_{\mathbb{F}},\mathcal{G}_{\mathbb{F}}) \rightarrow (\mathfrak{M}',\mathcal{F}') \otimes_{\mathcal{O}} \mathbb{F} \rightarrow 0	
\end{equation}
	in $\operatorname{Mod}_{\mathcal{F}}^{\operatorname{SD}}$. Assume that $\mu$ is pseudo-Barsotti--Tate and $V_{\mathbb{F}}\oplus V_{\mathbb{F}}'$ is cyclotomic-free. Then there exists a $G_K$-equivariant exact sequence
	$$
	0 \rightarrow V_{\mathbb{F}} \rightarrow W_{\mathbb{F}} \rightarrow V_{\mathbb{F}}' \rightarrow 0
	$$
	which identifies $\varphi,G_{K_\infty}$-equivariantly with \eqref{sequencetolift} after base-change to $W(C^\flat)$, and an $\mathcal{O}$-valued point $(\mathfrak{N},\beta,\mathcal{G}) \in \mathcal{L}^{(\mu,\mu'),\operatorname{conv}}_{R_{W_{\mathbb{F}}}}$ with $\mathfrak{N}$ fitting into a $\varphi$-equivariant exact sequence of $\mathfrak{S}_{\mathcal{O}}$-modules
	$$
	0 \rightarrow \mathfrak{M} \rightarrow \mathfrak{N} \rightarrow \mathfrak{M}' \rightarrow 0
	$$
	which recovers \eqref{sequencetolift} after applying $\otimes_{\mathcal{O}} \mathbb{F}$.
\end{lemma}
\begin{proof}
	The triple $(\mathfrak{M} \oplus \mathfrak{M}',\alpha \oplus \alpha',\mathcal{F}\oplus \mathcal{F}')$ can be viewed as an $\mathcal{O}$-valued point $x$ of $\mathcal{L}^{(\mu,\mu'),\operatorname{conv}}_{R_{V_{\mathbb{F}}\oplus V_{\mathbb{F}}'}}$. Let $x_{\mathbb{F}}$ be the composite $\operatorname{Spec}\mathbb{F}\rightarrow \operatorname{Spec}\mathcal{O}\xrightarrow{x}\mathcal{L}^{(\mu,\mu'),\operatorname{conv}}_{R_{V_{\mathbb{F}}\oplus V_{\mathbb{F}}'}}$. From \eqref{sequencetolift} we can construct an exact sequence
	$$
	0 \rightarrow (\mathfrak{M}\oplus \mathfrak{M},\mathcal{F}\oplus \mathcal{F}') \otimes_{\mathcal{O}} \mathbb{F}\rightarrow (\mathfrak{N}_{\mathbb{F}}^*,\mathcal{G}_{\mathbb{F}}^*) \rightarrow (\mathfrak{M}\oplus \mathfrak{M},\mathcal{F}\oplus \mathcal{F}') \otimes_{\mathcal{O}} \mathbb{F} \rightarrow 0
	$$
	recovering \eqref{sequencetolift} after pulling back\footnote{Note that pullbacks and pushouts exist in $\operatorname{Mod}_{\operatorname{\mathcal{F}}}$; the pushout of two morphisms $f:\mathfrak{M} \rightarrow \mathfrak{N}$ and $g:\mathfrak{M}\rightarrow \mathfrak{N}'$ is constructed as the cokernel of $(f,-g):\mathfrak{M} \rightarrow \mathfrak{N}\oplus\mathfrak{N}'$. Similarly the pullback of $f:\mathfrak{M} \rightarrow \mathfrak{N}$ and $g:\mathfrak{M}'\rightarrow \mathfrak{N}$ is constructed as the kernel of $f-g:\mathfrak{M} \oplus \mathfrak{M}' \rightarrow \mathfrak{N}$. It follows from Proposition~\ref{SDext} that these construction respect the full subcategory $\operatorname{Mod}^{\operatorname{SD}}_{\mathcal{F}}$.} along $(\mathfrak{M}',\mathcal{F}') \otimes_{\mathcal{O}} \mathbb{F} \rightarrow (\mathfrak{M}\oplus \mathfrak{M}',\mathcal{F}\oplus \mathcal{F}')\otimes_{\mathcal{O}}\mathbb{F}$ and then pushing out along $(\mathfrak{M}\oplus \mathfrak{M}',\mathcal{F}\oplus \mathcal{F}')\otimes_{\mathcal{O}}\mathbb{F}\rightarrow (\mathfrak{M},\mathcal{F}) \otimes_{\mathcal{O}} \mathbb{F}$. Corollary~\ref{surjective} implies that this new exact sequence arises from a tangent vector $t$ in $\mathcal{L}^{(\mu,\mu'),\operatorname{conv}}_{R_{V_{\mathbb{F}}\oplus V_{\mathbb{F}}'}}$ over the point $x_{\mathbb{F}}$.
	
	As $V_{\mathbb{F}}$ and $V_{\mathbb{F}}'$ are cyclotomic-free the same is true of $V_{\mathbb{F}}\oplus V_{\mathbb{F}}'$. Therefore Theorem~\ref{main} applies and the completed local ring of $\mathcal{L}^{(\mu,\mu'),\operatorname{conv}}_{R_{V_{\mathbb{F}}\oplus V_{\mathbb{F}}'}}$ is formally smooth over $\mathbb{Z}_p$. Hence $t$ can be lifted to an $\mathcal{O}[\epsilon]$-valued point (dual numbers over $\mathcal{O}$) inducing $\overline{x}$, cf. . Such a point gives rise to a $\varphi$-equivariant exact sequence
	$$
	0 \rightarrow \mathfrak{M} \oplus \mathfrak{M}' \rightarrow \mathfrak{N}^*\rightarrow \mathfrak{M} \oplus \mathfrak{M}' \rightarrow 0
	$$
	of $\mathfrak{S}_{\mathcal{O}}$-modules which $\varphi,G_{K_\infty}$-equivariantly identifies with an exact sequence of crystalline $G_K$-representations $0 \rightarrow V_{\alpha}\oplus V_{\alpha'} \rightarrow W^* \rightarrow V_{\alpha}\oplus V_{\alpha'}\rightarrow 0$ after base-changing to $ W(C^\flat)$. Pulling back along $\mathfrak{M}' \rightarrow \mathfrak{M}\oplus \mathfrak{M}'$ and then pushing out along $\mathfrak{M}\oplus \mathfrak{M}'\rightarrow \mathfrak{M}$ produces a $\varphi$-equivariant exact sequence
	$$
	0 \rightarrow (\mathfrak{M},\mathcal{F}) \rightarrow (\mathfrak{N},\mathcal{G}) \rightarrow (\mathfrak{M}',\mathcal{F}') \rightarrow 0
	$$
	and a $G_K$-equivariant exact sequence $0 \rightarrow V_{\alpha}\rightarrow W\rightarrow V_{\alpha'}\rightarrow 0$ of crystalline representations which $\varphi,G_{K_\infty}$-equivariantly identify after base-changing to $W(C^\flat)$. Since the formation of the pullbacks and pushouts commutes with $\otimes_{\mathcal{O}} \mathbb{F}$ we recover \eqref{sequencetolift} after basechanging to $\mathbb{F}$. Thus $W = R_{W_{\mathbb{F}}} \otimes_{\beta} \mathcal{O}$ for some $\beta$ and $(\mathfrak{N},\beta,\mathcal{G})$ defines an $\mathcal{O}$-point of $\mathcal{L}^{(\mu,\mu'),\operatorname{conv}}_{R_{W_{\mathbb{F}}}}$ as desired.
\end{proof}
\begin{lemma}
	Suppose that $V_{\mathbb{F}}$ is one-dimensional and that $(\mathfrak{M}_{\mathbb{F}},\alpha_{\mathbb{F}},\mathcal{F}_{\mathbb{F}})$ corresponds to an $\mathbb{F}$-valued $x_{\mathbb{F}}$ point of $\mathcal{L}^{\leq h_j,\operatorname{conv}}_{R}$ with $(\mathfrak{M},\mathcal{F}) \in \operatorname{Mod}^{\operatorname{SD}}_{\mathcal{F}}$ with pseudo-Barsotti--Tate Hodge type $\mu$.\footnote{We point out that every object of $\operatorname{Mod}_{\mathcal{F}}$ of rank one over $\mathfrak{S}_{\mathbb{F}}$ is contained in $\operatorname{Mod}^{\operatorname{SD}}_{\mathcal{F}}$.} Then there exists an $\mathcal{O}$-valued point $(\mathfrak{M},\alpha,\mathcal{F})$ of $\mathcal{L}^{\mu,\operatorname{conv}}_R$ inducing $x_{\mathbb{F}}$.
\end{lemma}
\begin{proof}
	The correspondence between Breuil--Kisin modules over $\mathfrak{S}_{\mathcal{O}}$ of rank one and crystalline characters is an equivalence. Moreover, in the rank one case the Hodge type of any such character is easily described in terms of the associated Breuil--Kisin module, cf. \cite[2.2.3]{GLS15}. It is therefore sufficient to observe that every rank one object in $\operatorname{Mod}_{\mathcal{F}}$ with a given Hodge type can be lifted to a rank one Breuil--Kisin module over $\mathcal{O}$$\mathfrak{S}_{\mathcal{O}}$ so that the Hodge type of the object in $\operatorname{Mod}_{\mathcal{F}}$ coincides with that of the corresponding crystalline character. This follows easily from the definitions.
\end{proof}
\begin{corollary}\label{liftingresult}
	Continue to assume $\mu$ is pseudo-Barsotti--Tate and $V_{\mathbb{F}}$ is cyclotomic-free. Suppose that every Jordan--Holder factor of $V_{\mathbb{F}}$ is one-dimensional and $(\mathfrak{M}_{\mathbb{F}},\alpha_{\mathbb{F}},\mathcal{F}_{\mathbb{F}})$ corresponds to an $\mathbb{F}$-valued point of $\mathcal{L}^{\mu,\operatorname{conv}}_R$. Then there exists an $\mathcal{O}$-valued point $(\mathfrak{M},\alpha,\mathcal{F})$ of $\mathcal{L}^{\mu,\operatorname{conv}}_R$ lying over $x_{\mathbb{F}}$ such that every Jordan--Holder factor of $V_{\alpha}[\frac{1}{p}]$ is one-dimensional.
\end{corollary}
\begin{proof}
	We induct on the dimension of $V_{\mathbb{F}}$. If $V_{\mathbb{F}}$ is one-dimensional there is nothing to prove. For the general case, any $G_K$-equivariant exact sequence $0\rightarrow V_{\mathbb{F},1} \rightarrow V_{\mathbb{F}} \rightarrow V_{\mathbb{F},2}\rightarrow 0$ induces a unique $\varphi$-equivariant exact sequence $0 \rightarrow \mathfrak{M}_{\mathbb{F},1} \rightarrow \mathfrak{M}_{\mathbb{F}}\rightarrow \mathfrak{M}_{\mathbb{F},2}\rightarrow 0$ which recovers the sequence of $G_K$-representations $\varphi,G_{K_\infty}$-equivariantly after base-change to $C^\flat$, cf. \cite[5.1.3]{B18}. By equipping $\mathfrak{M}_{\mathbb{F},i}$ with the appropriate filtrations we view this as a sequence in $\operatorname{Mod}_{\mathcal{F}}$. Using part (2) of Proposition~\ref{SDext} we see that if $\mathfrak{M}_{\mathbb{F},i}$ has Hodge type $\mu_i$ then $\mu = (\mu_1,\mu_2)$. Both $V_{\mathbb{F},i}$ are also cyclotomic-free and so our inductive hypothesis produces lifts of $\mathfrak{M}_{\mathbb{F},i}$. Using Lemma~\ref{liftingextensions} we obtain such a lift for $V_{\mathbb{F}}$ also.
\end{proof}
\subsection{Potential diagonalisability}

Let $V$ be a $G_K$-stable $\mathcal{O}$-lattice inside a crystalline representation of Hodge type $\mu$ and set $V_{\mathbb{F}} = V\otimes_{\mathcal{O}} \mathbb{F}$. Following \cite{BLGGT} we say $V$ is diagonalisable if $V$ lies on the same irreducible component of $\operatorname{Spec}R^\mu$ (equivalently the same irreducible component of $\operatorname{Spec}R^\mu[\frac{1}{p}]$) as an $E'$-valued point, for $E'/E$ finite, which is a direct sum of characters. We say $V$ is potentially diagonalisable if $V|_{G_{K'}}$ is diagonalisable for $K'/K$ some finite extension.

\begin{lemma}\label{pot}
	If $V[\frac{1}{p}]$ lies in the same irreducible component of $\operatorname{Spec}R^\mu$ as an $E'$-valued point whose corresponding representation admits a $G_K$-stable filtration with one-dimensional graded pieces then $V$ is potentially diagonalisable. 
\end{lemma}
\begin{proof}
	See \cite[2.1.2]{GL14}.
\end{proof}
\begin{corollary}
	If $\mu$ is pseudo-Barsotti--Tate and $V_{\mathbb{F}}$ is cyclotomic-free then $V$ is potentially diagonalisable.
\end{corollary}
\begin{proof}
	We can replace $V$ by $V|_{G_{K'}}$ for any finite extension $K'/K$. Since $V_{\mathbb{F}}$ is cyclotomic-free we can choose $K'$ so that $V_{\mathbb{F}}|_{G_{K'}}$ has one-dimensional Jordan--Holder factors and is also cyclotomic-free. Let $x$ be the $\mathcal{O}$-valued point of $\mathcal{L}^{\mu,\operatorname{conv}}_{R}$ corresponding to the $\mathcal{O}$-valued point of $\operatorname{Spec}R^{\mu}$ associated to $V$. Corollary~\ref{liftingresult} produces an $\mathcal{O}$-valued point $x'$ such that (i) $x$ and $x'$ coincide on the closed point of $\operatorname{Spec}\mathcal{O}$, and (ii) $\operatorname{Spec}\mathcal{O}\xrightarrow{x'} \rightarrow \mathcal{L}^{\mu,\operatorname{conv}}_{R} \rightarrow \operatorname{Spec}R^\mu$ corresponds to a deformation $V'$ with every Jordan--Holder factor of $V'[\frac{1}{p}]$ one-dimensional. Part (i) implies $x$ and $x'$ lie in the same connected component of $\mathcal{L}^{\mu,\operatorname{conv}}_R$, and so the same irreducible component in view of Theorem~\ref{main}. Hence $V$ and $V'$ lie on the same irreducible component of $\operatorname{Spec}R^\mu$. As $V'$ is potentially diagonalisable by Lemma~\ref{pot} so is $V$.
\end{proof}
\bibliography{/home/user/Dropbox/Maths/biblio.bib}

\begin{bibdiv}
\begin{biblist}

\bib{B18b}{article}{
      author={Bartlett, Robin},
       title={Potentially diagonalisable lifts with controlled {H}odge--{T}ate
  weights},
        date={2018},
     journal={Preprint, arXiv:1812.02042},
}

\bib{B18}{article}{
      author={Bartlett, Robin},
       title={Inertial and {H}odge-{T}ate weights of crystalline
  representations},
        date={2020},
     journal={Math. Ann.},
      volume={376},
      number={1-2},
       pages={645\ndash 681},
}

\bib{B19}{article}{
      author={Bartlett, Robin},
       title={On the irreducible components of some crystalline deformation
  rings},
        date={2020},
     journal={Forum of Mathematics, Sigma},
      volume={8},
       pages={e22},
}

\bib{BLGGT}{article}{
      author={Barnet-Lamb, Thomas},
      author={Gee, Toby},
      author={Geraghty, David},
      author={Taylor, Richard},
       title={Potential automorphy and change of weight},
        date={2014},
     journal={Ann. of Math. (2)},
      volume={179},
      number={2},
       pages={501\ndash 609},
}

\bib{CEGS19}{article}{
      author={Caraiani, Ana},
      author={Emerton, Matthew},
      author={Gee, Toby},
      author={Savitt, David},
       title={Moduli stacks of two-dimensional galois representations},
        date={2019},
     journal={Preprint, arXiv:1908.07019},
}

\bib{Fon00}{incollection}{
      author={Fontaine, Jean-Marc},
       title={Repr\'esentations {$p$}-adiques des corps locaux. {I}},
        date={1990},
   booktitle={The {G}rothendieck {F}estschrift, {V}ol.\ {II}},
      series={Progr. Math.},
      volume={87},
   publisher={Birkh\"auser Boston, Boston, MA},
       pages={249\ndash 309},
}

\bib{Gee06}{article}{
      author={Gee, Toby},
       title={A modularity lifting theorem for weight two {H}ilbert modular
  forms},
        date={2006},
     journal={Math. Res. Lett.},
      volume={13},
      number={5-6},
       pages={805\ndash 811},
}

\bib{GL14}{article}{
      author={Gao, Hui},
      author={Liu, Tong},
       title={A note on potential diagonalizability of crystalline
  representations},
        date={2014},
     journal={Math. Ann.},
      volume={360},
      number={1-2},
       pages={481\ndash 487},
}

\bib{GLS}{article}{
      author={Gee, Toby},
      author={Liu, Tong},
      author={Savitt, David},
       title={The {B}uzzard-{D}iamond-{J}arvis conjecture for unitary groups},
        date={2014},
     journal={J. Amer. Math. Soc.},
      volume={27},
      number={2},
       pages={389\ndash 435},
}

\bib{GLS15}{article}{
      author={Gee, Toby},
      author={Liu, Tong},
      author={Savitt, David},
       title={The weight part of {S}erre's conjecture for
  {$\operatorname{GL}(2)$}},
        date={2015},
     journal={Forum Math. Pi},
      volume={3},
       pages={e2, 52},
}

\bib{Kis06}{incollection}{
      author={Kisin, Mark},
       title={Crystalline representations and {$F$}-crystals},
        date={2006},
   booktitle={Algebraic geometry and number theory},
      series={Progr. Math.},
      volume={253},
   publisher={Birkh\"auser Boston, Boston, MA},
       pages={459\ndash 496},
}

\bib{Kis08}{article}{
      author={Kisin, Mark},
       title={Potentially semi-stable deformation rings},
        date={2008},
     journal={J. Amer. Math. Soc.},
      volume={21},
      number={2},
       pages={513\ndash 546},
}

\bib{KisFF}{article}{
      author={Kisin, Mark},
       title={Moduli of finite flat group schemes, and modularity},
        date={2009},
     journal={Ann. of Math. (2)},
      volume={170},
      number={3},
       pages={1085\ndash 1180},
}

\bib{Kis10}{article}{
      author={Kisin, Mark},
       title={Integral models for {S}himura varieties of abelian type},
        date={2010},
     journal={J. Amer. Math. Soc.},
      volume={23},
      number={4},
       pages={967\ndash 1012},
}

\bib{LLBBM20}{article}{
      author={Le, Daniel},
      author={Le~Hung, Bao~V.},
      author={Levin, Brandon},
      author={Morra, Stefano},
       title={Local models for galois deformation rings and applications},
        date={2020},
      eprint={arXiv:2007.05398},
}

\bib{PR09}{article}{
      author={Pappas, G.},
      author={Rapoport, M.},
       title={{$\Phi$}-modules and coefficient spaces},
        date={2009},
     journal={Mosc. Math. J.},
      volume={9},
      number={3},
       pages={625\ndash 663, back matter},
}

\bib{stacks-project}{misc}{
      author={{Stacks Project Authors}, The},
       title={\itshape stacks project},
         how={\url{http://stacks.math.columbia.edu}},
        date={2017},
}

\bib{Wang17}{article}{
      author={Wang, Xiyuan},
       title={{Weight elimination in two dimensions when $p=2$}},
        date={2017},
      eprint={arXiv:1711.09035},
}

\end{biblist}
\end{bibdiv}
\end{document}